\documentclass[letterpaper, reqno, 12pt]{amsart}
\DeclareFontFamily{U}{matha}{\hyphenchar\font45}
\DeclareFontShape{U}{matha}{m}{n}{
      <5> <6> <7> <8> <9> <10> gen * matha
      <10.95> matha10 <12> <14.4> <17.28> <20.74> <24.88> matha12
      }{}
\DeclareSymbolFont{matha}{U}{matha}{m}{n}
\DeclareFontFamily{U}{mathx}{\hyphenchar\font45}
\DeclareFontShape{U}{mathx}{m}{n}{
      <5> <6> <7> <8> <9> <10>
      <10.95> <12> <14.4> <17.28> <20.74> <24.88>
      mathx10
      }{}
\DeclareSymbolFont{mathx}{U}{mathx}{m}{n}

\DeclareMathSymbol{\obot}         {2}{matha}{"6B}
\DeclareMathSymbol{\bigobot}       {1}{mathx}{"CB}

\usepackage[pdfauthor={Congling Qiu}, 
        pdftitle={???}%
      dvips,colorlinks=true]{hyperref}
\hypersetup{
  bookmarksnumbered=true,
   linkcolor=red,
  citecolor=red,
  pagecolor=black, 
  urlcolor=black,  
}

\usepackage[utf8]{inputenc}

\usepackage{xtab}
\usepackage{amsmath, amssymb, cancel,verbatim,leftidx}
\usepackage[all]{xy}
\usepackage[scr=euler,scrscaled=.96]{mathalfa}
 
\usepackage{enumerate}
\usepackage{mathtools,booktabs}
\usepackage{amscd}
\usepackage{bbm, stmaryrd}
\usepackage{yfonts}
\usepackage{color}

\usepackage{epstopdf} 
\usepackage{booktabs}

\usepackage[nameinlink]{cleveref}

\newtheorem{teo}{Theorem}[subsection]
\newtheorem{thm}[teo]{Theorem}
\newtheorem{thm*}[teo]{Theorem*}

\newtheorem{lem}[teo]{Lemma}
\newtheorem{cor}[teo]{Corollary}
\newtheorem{defn}[teo]{Definition}

\newtheorem{rmk}[teo]{Remark}

\newtheorem*{hypo*}{Hypothesis}

  \newtheorem{prop}[teo]{Proposition}
    \newtheorem {conj}[teo]{Conjecture}

\numberwithin{equation}{section}

  \newcommand{\BA}{{\mathbb {A}}} 
    \newcommand{\BC}{{\mathbb {C}}} 
    \newcommand{\BE}{{\mathbb {E}}} 
     \newcommand{\BH}{{\mathbb {H}}}

    \newcommand{\BQ}{{\mathbb {Q}}} \newcommand{\BR}{{\mathbb {R}}}

     \newcommand{\BZ}{{\mathbb {Z}}}

     \newcommand{\cB}{{\mathcal {B}}}

     \newcommand{\cH}{{\mathcal {H}}}

    \newcommand{\cM}{{\mathcal {M}}} 
    \newcommand{\cO}{{\mathcal {O}}}

    \newcommand{\RM}{{\mathrm {M}}}

    \newcommand{\fc}{{\mathfrak{c}}}

     \newcommand{\fn}{{\mathfrak{n}}}
     \newcommand{\fp}{{\mathfrak{p}}}

    \newcommand{\wt}{\widetilde}\newcommand{\ol}{\overline}
    \newcommand{\wh}{\widehat}

    \newcommand{\incl}{\hookrightarrow}

    \newcommand{\bsl}{\backslash}

  \newcommand{\ep}{\epsilon}
  \newcommand{\vpl}{\varprojlim}
             
 \newcommand{\vil}{\varinjlim}  
 \newcommand{\lb}{\left(} \newcommand{\rb}{\right)}

    \newcommand{\Aut}{{\mathrm{Aut}}}

    \newcommand{\Ch}{{\mathrm{Ch}}}

    \newcommand{\Cl}{{\mathrm{Cl}}}

    \newcommand{\End}{{\mathrm{End}}}

    \newcommand{\Gal}{{\mathrm{Gal}}} \newcommand{\GL}{{\mathrm{GL}}}

    \newcommand{\Hom}{{\mathrm{Hom}}}
    
    \newcommand{\id}{{\mathrm{id}}}

        \newcommand{\DC}{{\mathrm{DC}}}

    \newcommand{\Jac}{{\mathrm{Jac}}} \newcommand{\JL}{{\mathrm{JL}}}%\newcommand{\ker}{{\mathrm{Ker}}}

     \newcommand{\Pic}{\mathrm{Pic}}

\newcommand{\rf}{{\mathrm{f}}}

    \newcommand{\Spec}{{\mathrm{Spec}}}

    \newcommand{\vol}{{\mathrm{vol}}}

      \newcommand{\pr}{{\mathfrak{pr}}}

\newcommand\supervisor[1]{\def\@supervisor{#1}}

\newcounter{elno}

\renewcommand{\cong}{\simeq}

\begin{document} 
  \title{ Vanishing results for the modified diagonal cycles II: Shimura curves}
  
 %  \author{Congling Qiu,\ Wei Zhang}

\author{Congling Qiu}
\address[Congling Qiu]{Department of Mathematics, MIT}
\email{qiuc@mit.edu}

\author{Wei Zhang}
\address[Wei Zhang]{Department of Mathematics, MIT}
\email{weizhang@mit.edu}

\begin{abstract}
We prove vanishing results for the modified diagonal cycles in the Chow groups of the triple products of Shimura curves and their motivic direct summands. In particular we find examples of curves with trivial automorphism groups and vanishing modified diagonal cycles.
\end{abstract} 
\maketitle 
 \tableofcontents

  \section{Introduction}

 \subsection{}\label{vanish}

  For a smooth projective  variety $X$ over a  field $F$,   let $\Ch^i(X)$ be the Chow group  of $X$   of  codimension $i$ cycles with   $\BQ$-coefficients, modulo rational equivalence.  For an algebraic cycle $Z$ on $X$, let $[Z]$  denote the class of  $Z$  in $\Ch^i(X)$.
  The diagonal cycle $\Delta$ on $X^n=X\times\cdots\times X$ ($n$-times) 
  has been studied by many authors. Of particular interest is the modified diagonal cycle in the case $n=3$ and $X$ is a curve (always assumed to be connected), which is  called the Gross--Schoen cycle \cite{GS,Zha10}.    
  When $X$ is a hyperelliptic curve, Gross and Schoen \cite{GS}
 showed that  the  diagonal cycle  modified by a  Weierstrass point vanishes in $\Ch^2(X^3)$. It was once  suspected  that over $\BC$, hyperelliptic curves  are the only curves  of genus $>1$ with vanishing  modified diagonal cycles. 
In  our previous paper \cite{QZ1} and the paper \cite{Lat} by  Laterveer,  certain low-genus examples of non-hyperelliptic $X/\BC$  are proven to have vanishing  modified diagonal cycles in Chow groups. A  special feature of these examples is that they have
 large automorphism groups.  
 Vanishing results in certain quotients of $\Ch^2(X^3)$ were also obtained by other authors, see for example Beauville--Schoen \cite{BS}, as well as the introduction of \cite{QZ1}.)
   In this paper, we give five    examples of plane quartics  over $\BC$  (i.e., non-hyperelliptic of genus $3$)   with vanishing  modified diagonal cycles and   trivial automorphism groups   (see \S\ref {curves of genus 3}).    In particular,  they are neither  curves in our previous work \cite{QZ1}, nor the curves of Laterveer \cite{Lat} and Beauville--Schoen \cite{BS}.

  In \cite[Conjecture 1.4.1]{QZ1}, we also conjectured that there are  non-hyperelliptic curves over $\BC$  of arbitrary large genera  with vanishing  modified diagonal cycles.  
This seems out of reach at this moment. Instead we may ask a similar question for the projection of the diagonal cycle to the  triple product of an effective submotive $M$  of (the Chow motive of)  a curve  $X$. We  are led to a (partial) analog of the aforementioned conjecture, see Conjecture \ref{conj:1}. Recall a  submotive $M=(X,p)$ of $X$  is given by a projector (i.e. idempotent)  $p\in   \Ch^1(X^2)$, viewed as the ring of self-correspondence on the curve $X$.   Its  genus  is defined  to be $$g(M):=\dim p_* H^0(X,\omega_X)/2$$ % where $\omega_X$  is the canonical bundle on $X$.
     We   define its diagonal  cycle   to be 
      \begin{align}\label{eq:mot Del}
[\Delta]_{M}:=p^{\otimes 3}_*[\Delta]\in \Ch^2(X^3).
\end{align}
  We especially focus on   isotypic submotives of   Shimura curves.
 We show some nice vanishing properties of their diagonal cycles: 
 vanishing in  arbitrary large genera (Theorem     \ref{thm:2}, note that this is not implied by  the last conjecture); vanishing according to nonvanishing of central $L$-values (Theorem     \ref{thm:3}); 
 vanishing for all but at most one class of  Shimura curves (Theorem     \ref{thm:4}).

     \subsection{}\label{1.2}

 Now we give more details of our results. 
For $e\in \Ch^1(X)$ with $\deg e=1$, the modified diagonal cycle   $[\Delta_{e}]\in \Ch^2(X^3)$ is defined as follows. 
  If $e=[p]$ the class of an $F$-point $p$, we set 
\begin{alignat*}{3}\Delta_{12}&= \{(x,x,p):x\in X\},\quad \Delta_{23}&&= \{(p,x,x):x\in X\}, \quad
\Delta_{31}& &= \{(x,p,x):x\in X\}, \\
\Delta_{1}&=  \{(x,p,p):x\in X\},\quad \ 
\Delta_{2}&&=   \{(p,x,p):x\in X\},\quad  \
\Delta_{3}& &= \{(p,p,x):x\in X\}.
\end{alignat*}
Then $[\Delta_{e}]$ is  the   class  in $\Ch^2(X^3)$ of the algebraic cycle
\begin{align*} 
\Delta-\Delta_{12}-\Delta_{23}-\Delta_{31}+
\Delta_{1}+\Delta_{2}+
\Delta_{3}.
\end{align*}
 In general,   let      $\delta:X\to X^2$ be the   diagonal embedding, and we define via \eqref{eq:mot Del}
  $$[\Delta_{e}]=[\Delta]_{h^{\geq 1}(X)}$$ 
 where  $h^{\geq 1}(X):=(X, [\delta_* X]- [X]\times e)$. 
  In   \cite{QZ1},  we also considered
$$
h^{1}(X):=(X,\delta_{e})
$$
  with
   \begin{equation}
  \label{delta}
\delta_{e}:=[\delta_* X]-(e\times [X] +[X]\times e) .
 \end{equation}  
We showed   that 
 \begin{equation}
  \label{eq:equiv}
  [\Delta_{e}]=0\text{ if and only if }[\Delta]_{h^{ 1}(X)}=0.
   \end{equation}

By  \cite{GS}, we have $[\Delta_{e}]=0$ if $g(X)\leq 1$.
Since the base change to the separable closure of $F$ does not alter the vanishing/nonvanishing of $[\Delta_{e}]$, we have $[\Delta_{e}]=0$ if its geometrically connected components have genera $\leq 1$.
 Assume that geometrically connected components  of $X$
 have genera  $ \geq 2$.    In our previous  paper \cite{QZ1}, an elementary argument  shows that    $[\Delta_{e}]=0$ only if  
 \begin{align}\label{eq:xi}
e=\frac{1}{\deg \omega_X} \omega_X\in \Ch^1(X).
\end{align}  
 However,  different choices of $e$ give isomorphic $h^{\geq 1}(X)$ (in the category of motives).  
In particular, for a motive $M$ that can be realized as a submotive    of a curve,  vanishing  of  $[\Delta]_M$ depends on the  realization of $M$ as a submotive, not just   its isomorphism class.
 A less elementary example is given by Theorem \ref{thm:2}  below (and the discussion following it).

  \subsection{}

  In this paper,  we   consider  submotives of  $h^{ 1}(X)$ which have  the following     nice description. Recall  the
 well-known  isomorphism (see \cite[Proposition 3.3]{Sch})
 $$
 \End \lb h^{ 1}(X) \rb\cong \End (J_X)_\BQ
 $$ 
where an endomorphism of $h^{ 1}(X)$  is  sent to the induced morphism.
   Here  and below, $J_X$
 is the Jacobian variety of $X$.
The  $\BQ$-algebra $\End(J_X)_\BQ$ is essentially  a direct sum of matrix algebras indexed by 
simple isogeny factors of $J_X$
 (see \eqref{eq:EndJ}). 
For a  simple abelian variety $A$,  let $\delta_A=\delta_{X,A}\in \End \lb h^{ 1}(X) \rb$ be the corresponding projector (which is 0 if $A$ is not an isogeny factor of $J_X$). 
For example, if $J_X$ is isogenous to a power  of $A$, then $\delta_A=\delta_e$.
Define the $A$-isotypic component of  $h^{ 1}(X) $ to be
 $$h^A(X):=(X, \delta_A).$$

 Partly motivated by  \cite[Conjecture 1.4.1]{QZ1} and the fact that every abelian variety over $\BC$ is an isogeny factor of the Jacobian variety of some curve over $\BC$, we are led to the following conjecture. 
\begin{conj}\label{conj:1}
 Let $A/\BC$ be  a  simple abelian variety. 
 There exists a  curve $X/\BC$   such that  $h^A(X)\neq 0$ and
$[\Delta]_{h^A(X)}=0.$  Moreover, among  non-hyperelliptic $X$  with $[\Delta]_{h^A(X)}=0$,  
 $g\lb h^A(X)\rb$  is unbounded.
        \end{conj}
       The first part of the conjecture seems already out of reach. 
By \cite{NP}, a very general abelian variety $A$ of dimension $\geq 3$ is not dominated by a hyperelliptic Jacobian. Thus for such $A$, one cannot expect to use   hyperelliptic curves in the place of $X$ to easily prove 
  the first part of the conjecture. 
 Moreover, this fact from \cite{NP} ensures that the conjectural vanishing
 $[\Delta]_{h^A(X)}=0$  cannot be related to a correspondence from a hyperelliptic curve to $X$.

The first  result of our paper proves this  conjecture for certain elliptic curves.     \begin{thm}\label{thm:2} 
   Let   $E$   be an elliptic curve  over $\BQ$ with non-integral $j$-invariant.  Then  Conjecture \ref{conj:1} holds for $E_\BC$.

        \end{thm}

      Note that,  though   $[\Delta]_{h^{ 1}(E)}=0$ for    any elliptic curve $E$, the class   $[\Delta]_{h^E(X)}$ may be nonvanishing for some $X$. Some examples were already given by Top  \cite[Theorem 3.4.2]{Top}.

The above  theorem is proved in the end of
          \S \ref{ss:thm2}. In fact, it is deduced from more general result (Corollary \ref{cor:2+}) for modular abelian varieties. Moreover, 
  the analog of Theorem \ref{thm:2} holds for any elliptic curve $E$ over the  function field $F$ of a curve $C$ over a finite field, see Remark \ref{rmk fun fd}.

   \subsection{} 
  
     For  a triple of submotives $ M_i=(X,p_i), i=1,2,3$ of $X$,       define  $$[\Delta]_{M_1\otimes  M_2\otimes M_3}:= \lb p_1 \otimes  p_2\otimes  p_3\rb  _*[\Delta]\in \Ch^2(X^3).$$
%     When  $X$ is some  modular curve,
%  Gross and Kudla  \cite{GK} considered $ [\Delta]_{M_1\otimes  M_2\otimes M_3}$ modulo the kernel of the Beilinson--Bloch height pairing \cite{Bei,GS} and its relation with triple product L-functions.
%     
%%Let $F$ be a number field.  We also prove results to relate the vanishing of modified diagonal cycles to the non-vanishing of $L$-functions. 
%     \WZ{I donot see the relevance of G-S work here? When  $X$ is some  modular curve,
%  Gross and Kudla  \cite{GK} considered $ [\Delta]_{M_1\otimes  M_2\otimes M_3}$ modulo the kernel of the Beilinson--Bloch height pairing \cite{Bei,GS}.}
% 

We propose the following conjecture, which is a special case of the Beilinson--Bloch 
conjecture   \cite[Conjecture 5.0]{Bei}.
   \begin{conj}\label{conj:2} Let $F$ be a number field, and let $X$ be a curve over $F$. Let $A_1,A_2,A_3$ be  three simple isogeny factors of $\Jac_X$. Assume that the L-function $L\lb s,h^1(A_1)\otimes h^1(A_2)\otimes h^1(A_3)\rb$ has holomorphic continuation to $s\in\BC$. Then if 
   $$L\lb 2,h^1(A_1)\otimes h^1(A_2)\otimes h^1(A_3)\rb \neq 0, $$  
    we have $$[\Delta]_{h^{A_1}(X)\otimes h^{A_2}(X)\otimes h^{A_3}(X)}=0.
     $$ 
    
         \end{conj}
     
      In particular, for a  curve $X$ over $F$ such that        such that $L\lb s,h^1(X)^{\otimes 3}\rb $ has holomorphic continuation to $s\in\BC$ and 
  $L\lb 2,h^1(X)^{\otimes 3}\rb \neq 0,$
  we have $[\Delta_e]=0$.

For a quaternionic Shimura curve, the holomorphic continuation assumption in    Conjecture \ref{conj:2} holds (Corollary \ref{holom}).
     \begin{thm}\label{thm:3} 
   Conjecture \ref{conj:2} holds if $X$ is a quaternionic Shimura curve over a totally real numbr field.
        \end{thm}
                 
In fact  we give a representation theoretical condition  (Corollary \ref {generalcor},
 a special case of  Theorem \ref{general}) for the  vanishing of the  diagonal cycles on an arbitrary power of  a Shimura curve, in terms of  the  vanishing of 
  diagonally invariant multi-linear forms on automorphic representations.
   When  specialized to the triple power,   Prasad's dichotomy    \cite[Theorem 1.2]{Pra} (see also Theorem  \ref{thm:Pra} (1)). 

    \begin{thm}\label{thm:4} 
 For  elliptic curves $E_i,i=1,2,3$   over a  totally real number field $F$,  for all but at most one  quaternion algebras over $F$,  $[\Delta]_{{h^{E_1}(X)\otimes h^{E_2}(X)\otimes h^{E_3}(X)}}= 0$ for  all of their associated 
  Shimura curves $X$.

           \end{thm}  
           Both the above two theorems are proved in \S           \ref{Triple product diagonal cycle}.
     In short,    we use the fact that the representation theoretical  condition  is further related to local root numbers of the $L$-function (see Proposition \ref{general3}).
   
          \subsection{} A key observation for the above results is the automorphy of Chow groups observed in \ref         {Product of Shimura curves}.
          Let $V$ be the product of $n$ (not necessarily isomorphic)  quaternionic Shimura curves over a totally real field. Then
 $ \Ch^\ast(V )$,  as a  module over the Hecke algebra,  is a direct sum of the ones appearing in the cohomology of $V$.
 This is also predicted by
          the conjectural injectivity of the Abel--Jacobi map  \cite{Bei}.
                 
\subsection*{Acknowledgment}
%The work started from an attempt to answer a question of Dick Gross in his talk at the Kudla Fest in 2021. 
 C. Q.  is partially supported by the NSF grant DMS \#2000533.
 W. Z. is partially supported by the NSF grant DMS \#1901642 and the Simons foundation.

 \subsection*{Notations and conventions}  
 
%Chow groups always  have  coefficients in $\BQ$, unless   otherwise stated.\WZ{We said so in the beginning of Intro}

 By a cycle on a variety $X$ we mean an element in the Chow group $\Ch^\ast(X)$.
 
For a morphism $f:V\to U$, let $\Gamma_f$ be the graph  of $f$, as an algebraic cycle on $V\times U$.
Denote the identity morphism on $V$ by $\id_V$. Then $\Gamma_{\id_V}$ is the  diagonal algebraic cycle of $V^2$.

For varieties $V,U$ over $F$,  cycles  $z_1$ on $ V $ and $z_2$ on $ U$, $z_1\times z_2$ is the product cycle on $V\times U$.
This defines the natural map $\Ch^*(V)\otimes\Ch^*(U)\to \Ch^*(V\times U)$.
Then $z_1\otimes z_2$ is understood as  $z_1\times z_2$.

 For varieties $V,U,W$ over $F$,    cycles $z_1$ on $ V\times U$ (understood as a correspondence from $V$ to $U$)
  and $z_2$ on $ U\times W $, let
$z_2\circ z_1$ be the correspondence composition, which is a cycle on $V\times W$. 

%\WZ{Do we view $z_2\circ z_1$ as an element in Chow group? In Deninger--Murre the identity is understood as  }
%\qcl{qcl: It is a cycle on $V\times W$, by convention?}
 
 For a   cycle $z$ on the square $V^2$ of a variety $V$, let $z^t$  be  the transpose of $z$.

\section{Motivic decomposition  of a curve}\label{An alternative approach}

In this section, we decompose the ``identity"  divisorial correspondences on a curve using simple isogeny factors of its Jacobian.
With this decomposition,  we derive vanishing  criteria for     cycles  on product of curves.
\subsection{Divisorial correspondences}\label{Divisorial correspondences}

  Let $X,X'$ be   smooth   projective  varieties over an arbitrary   field $F$.    
 Fix $e \in \Ch_0(X) ,e'\in \Ch_0(X)$    of degree 1. 
 \begin{defn}
  Let $\DC(X\times X')$ be the subspace of  $\Ch^{1}(X\times X')$ of $z$'s such that $z_* e\in \BQ e'$ and $z^*e\in \BQ e'$.  
  Let $\DC^1(X\times X')$ be the subspace of  $\Ch^{1}(X \times X' )$ of $z$'s such that $z_* e=0$ and $z^*e'=0$.  
  \end{defn}
   We call  cycles  in $\DC^1(X\times X')$  divisorial correspondences with respect to  $e,e'$.
 %Thus,  cycles  in $\DC(X^2)$   are called generalized   divisorial correspondences. 

  Assume that $X$ is a curve. % with  the Jacobian variety $J$. 
   By  \cite[Theorem 3.9]{Sch},  the following natural map  induced by pushforward  of divisors
 induces an isomorphism  of $\BQ$-vector spaces
\begin{equation}\label{Sch3.9}
 \DC^1(X\times X')\cong \Hom(J_X,\Pic^0_{X'/F})_\BQ.
 \end{equation}

  Consider the case $X'=X,e'=e$.
We may view $\Ch^1(X^2)$ as the (non-commutative) ring of self-correspondence on the curve $X$. Then $\DC(X^2)$ is a subring.  
The identify is  the diagonal  cycle  $[\Gamma_{\id_X}]$.  
Under correspondence composition, $\DC^1(X^2)$ is also a ring. 
The identify  is    the 1st Chow--K\"unneth projector
\begin{equation}
  \label{delta}
\delta_{X}:=[\Gamma_{\id_X}]-(e\times [X] +[X]\times e).
 \end{equation} 
This fact is a special case of the following   more general lemma.
   For $z\in \DC(X^2)$, we assign 
$$
 \lb    z^*e \times [X],  z-z^*e  \times [X]-[X]\times  z_* e,
[X]\times   z_* e\rb
$$ 
to get the following direct sum decomposition (the sum is clearly direct):     \begin{equation}\label{refm2}
    \DC(X^2)=\BQ(e\times [X]) \oplus  \DC^1(X^2)\oplus  \BQ([X]\times e).
    \end{equation}
       It is straightforward to check the following lemma.
 \begin{lem} \label{zdelta}The left/right composition by $\delta_X$ on  $\DC(X^2)$  is the projection to $\DC^1(X^2)$ via \eqref{refm2}.
 \end{lem}

%In this paper, we always consider abelian varieties up to isogeny. 
 
 \subsection{Endomorphisms of   abelian varieties}\label{Decomposition of the identity}
  Let $J$ be   an abelian variety over an arbitrary field $F$.
  % Assume that there is an isomorphism  $J\cong J^\vee$ and fix this isomorphism.
Then in the isogeny category of abelian varieties, we have 
\begin{equation}
\label{JJpi}
J\cong \prod_{A} A^{d_A} \end{equation}
where $A$ runs over a set of representatives of  simple isogeny factors of $J$. 
 Let 
    $M_A=\End(A)_\BQ$ which is a finite-dimensional  division algebra over $\BQ$,
  $\pi_A=\Hom(J,A)_\BQ$ and $\wt \pi_A=\Hom(A,J)_\BQ$  so that   
 $\pi_A$ (resp. $\wt \pi_A$) is a   left (resp. right) $M_A$-module of rank $d_A$  where the action is by post-composition (resp. pre-composition).   Define a pairing 
  \begin{equation}\label{duality}
  \pi_A
 \times\wt \pi_A\to M_A,\quad  (\phi,\psi)\mapsto\phi\circ \psi 
  \end{equation}
 %where $\phi^\vee \in \Hom(A,J)$ is the dual homomorphism.
  It is direct to check that this   pairing  is linear  for  the left (resp. right) $M_A$-action on  
  $  \pi_A$ (resp. $\wt \pi_A$) and left (resp. right) $M_A$-multiplication on
  $M_A$,  and it is perfect. Thus $\wt \pi_A$ is the dual $M_A$-module of $\pi_A$.

By   \eqref{JJpi}, the  
\begin{equation}\label{eq:EndJ}
\End(J)_\BQ\cong   \bigoplus_{A} M_A\otimes_\BQ \RM_{d_A}(\BQ),
  \end{equation} 
   where the latter is the matrix algebra. In particular,  it is semisimple and $\pi_A$'s consist of all simple right-modules of $\End(J)_\BQ$ (and the right action is by pre-composition).
  We have the canonical isomorphism
  \begin{equation}\label{pipiJ}
  %\begin{split}
  \bigoplus_{A}  \pi_A\otimes_{M_A} \wt \pi_A\cong \End(J)_\BQ, \
  \phi\otimes \psi\mapsto  \psi\circ \phi
  %\end{split}
  \end{equation} 
  Let $\cB_A$ be an $M_A$-basis of $\pi_A$, and for $\phi\in \cB_A$, let $\wt \phi$ be the dual element in the dual basis.
Then  
 the $\pi_A$-isotypic  projector  is    \begin{equation}\label{decomproj}
 \sum_{\phi\in \cB_A} \wt\phi \circ \phi \in \End(J)_\BQ.
\end{equation}
 Equivalently,  under \eqref{JJpi}, \eqref{decomproj}  is identity on the $A$-component, and 0 on other components.

 \subsection{Projectors for a curve I}
Let $J=J_X$ be the Jacobian variety of the curve $X$.  
For a  simple isogeny factor $A$,  let $ \delta_{A}=\delta_{X,A} \in\DC^1(X^2)$ be the  $\pi_A$-isotypic  projector under  \eqref{Sch3.9} and \eqref{eq:EndJ}.
 
 We first consider base change.
 Let $E/F$ be a field extension and $X^o$ a connected component of $X_E$.
 We have a commutative diagram 
  \begin{equation*} 
\xymatrix{ \DC^1(X^2)  \ar[r]\ar[d] &\End(J)_\BQ\ar[d] \\
 \DC^1(X^{o,2})\ar[r]
&\End(J^o)_\BQ},
 \end{equation*}
   where the horizontal maps are isomorphisms \eqref{Sch3.9} and vertical maps are restrictions.
 Let    $A^o$ be a simple isogeny factor of $A_E$ appearing in $J^o=J_{X^o}$. Then from the diagram, it is easy to see  that
   \begin{equation}\label{eq:bc}
  \delta_{A^o} \circ \lb \delta_{A}|_{X^{o,2}}\rb=  \delta_{A^o} .
 \end{equation}
 
 Now  we prove a proposition, which refines \cite[Proposition 2.3.3]{QZ1}.
 And it will be used in  \S \ref{curves of genus 3}.
 Let     $X'$ also  be a  curve  and $ e'\in \Ch_0(X)$    of degree 1. 
  Define  $\delta_{X'}$ as in \eqref{delta} accordingly.
 Let $A_i$, $i=1,...,m$, be all the  simple   isogeny  factors of  $J'=J_{X'}$. 
To  simplify the notations,   let
 $\delta_i=\delta_{{X,A_i}}$ and $\delta'_i= \delta_{X',A_i}$.
 
 \begin{lem}\label{mterms}For $z\in \Ch^1(X\times X')$, we have $$
\delta'_{i}\circ z\circ\delta_X  = \delta'_{i}\circ z\circ  \delta_{i}
=\delta_{X'}\circ z\circ  \delta_{i}.$$

 \end{lem} 
  \begin{proof}  It is direct to check that both sides of the equation 
 are in $\DC^1(X^2)$.
 By 
 \eqref{Sch3.9}, 
one can check the  equation via pushforwards along the  diagram 
$J\to J\to J'\to J'.$
 %Note that  $$ (f_i\times f_j\times f_k)_\ast(\Delta_\xi(X))= (f_i\times f_j\times f_k)_\ast(\Delta(X)).$$
  \end{proof}
\begin{prop}\label{vancor1}  
Assume that there is a morphism $f:X\to X'$ such that $e'=f_* e$. 
 Let $\Delta$ (resp. $\Delta'$) be the diagonal embedding of $X$ in $ X^n$ (resp. $X'$ in $ X'^n$).

(1) For   an $n$-tuple $(i_1,...,i_n) $ in $\{1,...,m\}$,   if 
  $\lb \delta_{ {i_1}}\otimes...\otimes \delta_ {i_n}\rb_* [\Delta]=0$, then we have $\lb \delta'_{ {i_1}}\otimes...\otimes \delta'_ {i_n}\rb_* [\Delta']=0$.
In particular,   if 
  $\lb \delta_{ {i_1}}\otimes...\otimes \delta_ {i_n}\rb_* [\Delta]=0$ for   every $n$-tuple $(i_1,...,i_n) $ in $\{1,...,m\}$, then
$\lb \delta_{X'}^{\otimes n}\rb_*[\Delta']=0$.

(2) Conversely, assume $\Hom (A_i,J/J')=0.$ 
   If 
 $\lb \delta'_{ {i_1}}\otimes...\otimes \delta'_ {i_n}\rb_* [\Delta']=0$ for   an $n$-tuple $(i_1,...,i_n) $ in $\{1,...,m\}$, then   $\lb \delta_{ {i_1}}\otimes...\otimes \delta_ {i_n}\rb_* [\Delta]=0$  .

\end{prop} 
  \begin{proof}  It is direct to check that   $e'=f_* e$ implies
  $$ \Gamma_f\circ \delta_X= \delta_{ X'} \circ \Gamma_f.$$ Then 
  $(\Gamma_f\circ\delta_X)^{\otimes n}_*[\Delta]=( \delta_{X'})^{\otimes n}_*[\Delta']$. 
  As $ \delta_{X'}$ is the multiplicative identity of $\DC^1(X'^2)$, it follows
   $$\lb \delta'_{ {i_1}}\otimes...\otimes \delta'_ {i_n}\rb_* [\Delta']={\otimes }_{j=1}^n(\delta'_{ {i_j}} \circ\Gamma_f\circ\delta_X)_*[\Delta].$$ Now apply Lemma \ref{mterms}.  
Then    we have  $$\lb \delta'_{ {i_1}}\otimes...\otimes \delta'_ {i_n}\rb_* [\Delta']=\lb \gamma_{f,i_1} \otimes...\otimes \gamma_{f,i_n}\rb_*
 \lb \delta_{ {i_1}}\otimes...\otimes \delta_ {i_n}\rb_* [\Delta],$$
 where $\gamma_{f,i}:=\delta'_{i}\circ \Gamma_{f}\circ  \delta_{i}$.
 (1) follows. In (2), claim: there exists $z_i\in \DC^1(X'\times X) $ such that 
 $z_i\circ \gamma_{f,i}=\delta_i.$  This can be checked in $\Hom(J,J)_\BQ$ along $J\to J'\to J$. 
 Then apply $\lb z_{ {i_1}}\otimes...\otimes z_ {i_n}\rb_* $ to the last equation. (2) follows.
    \end{proof}
   \begin{rmk}The last claim says that $\gamma_{f,i}:h^{A_i}(X)\to h^{A_i}(X')$ is an isomorphism.  
      \end{rmk}

  \subsection{Deninger--Murre projector}
 Let $A$ be an abelian variety (not necessary simple) over $F$ of dimension $g$ with identity $u$, and dual abelian variety $A^\vee$. For $n\in \BZ$, let   ${n_A}$ be the  ``multiplication by $n$" morphism on $A$.
  Let $\pr \in \Ch^g(A\times A)$  be the $(2g-1)$-th Deninger--Murre projector, that is denoted by $\pi_{2g-1}$ in Theorem \ref{Deninger--Murre projectors}.  
 A
simple but illustrative example occurs when
 $g=1$. Then   $\pr=\delta_{X}$ as  defined in \eqref{delta} with $X=A$ and $e=u$.   %We do not need other properties and general Deninger--Murre projectors,  we refer to \cite{DM} for these remarkable projectors.

For a variety $X$ over $F$, the abelian group structure of $A$ induces an abelian  group structure on $  \Hom(X,A) $, the set of morphisms of $F$-varieties. 
Regarding $X\times A$ as an abelian scheme over $X$, Corollary \ref{Deninger--Murre projectors2} implies  that the following map  is a homomorphism of abelian groups:  $$ \Hom (X,A) \to  \Ch_{\dim X}(X\times A),\quad f\mapsto {\pr}\circ [\Gamma_f].$$  

 \begin{defn}\label{prf} Denote  the unique  $\BQ$-linear extension $ \Hom (X,A)_\BQ \to  \Ch_{\dim X}(X\times A)$ by $f\mapsto {\pr_f}$.

\end{defn}
 Fix $e \in \Ch_0(X)  $    of degree 1.
\begin{defn}[\cite{YZZ}]\label{Hom0} 
Let $\Hom^0(X,A)$ be the subspace of  $f\in \Hom(X,A)_\BQ $  such that
for some (equivalently for all) representative $\sum_i a_i p_i$ of $e_{\ol F}$, where $a_i\in \BQ$ and $p_i\in X( \ol F)$, we have $\sum_i a_if(p_i)=u$ in $A(\ol F)_{\BQ}$. Here $p_i$ and $u$ are points rather than cycles, and  the addition is the one on $A$ but not for cycles. \end{defn}

Below, assume that  $X$ is a curve.
 \begin{lem}\label{alg0}
For $f\in \Hom^0(X,A) $ and $z\in\Ch^{1}(X\times A)$, we have $z^*   \pr_{f, *} e =0$ in $\Ch^1(X)$.
\end{lem}
\begin{proof}We may assume that $f\in \Hom(X,A)$. Then we have  $\pr_{f, *} e={\pr}_*(f_*e)$.
By Corollary \ref{Deninger--Murre projectorscor} (2), ${\pr}_*(f_*e)={\pr}_*(f_*e-[u]).$
By definition, $ (f_*e-[u])$ is in the kernel of the Abel--Jacobi map. So $z^*\lb{\pr}_*(f_*e)\rb=0$ is in the kernel of the Abel--Jacobi map  for $X$ (by the functoriality of the Abel--Jacobi map), which is 0.  
\end{proof}
Let $ \DC^1(X\times A) $ be  defined as in \ref{Divisorial correspondences} with $X'=A,e'=u$.

  \begin{cor}\label{trans?}
   For $f  \in \Hom^0(X,A)$ and $z\in  \DC^1(X\times A) $,  we have $  \pr_f^t\circ z \in \DC^1(X^2).
 $  

 \end{cor}
  \begin{proof}
  By definition,  $(\pr_f^t\circ z)_{*}e=0$.
 By   Lemma \ref{alg0}, we have  $
(\pr_f^t\circ z)^{*}e= z^\ast (\pr_{f,\ast} e)=0.
$  \end{proof}
\subsection{Projectors for a curve II}
Let $J=J_X$ be the Jacobian variety of $X$ with the canonical principal polarization $J\cong J^\vee$.
Taking dual homomophisms defines
\begin{equation}\label{dualh}
\Hom(J,A)\cong \Hom(A^\vee,J).
\end{equation} 
 
The canonical embedding $X\incl \Pic^1_{X/F}$ and $e\in \Ch^1(X)$ define an   element in $\Hom^0(X,J)$.
The Albanese property of the Jacobian variety then induces an isomorphism 
\begin{equation}\label{alb}
\Hom^0(X,A)\cong\Hom(J,A)_\BQ.
\end{equation}

   \begin{prop}\label{trans} (1) The following natural maps  induced by pushforward, pullback, and pushforward, respectively,  of divisors
 $$
 \DC^1(X\times A)\to \Hom(J,A^\vee)_\BQ,
 $$
 $$
\Hom^0(X,A)\to \Hom(A^\vee, J)_\BQ,
 $$
   $$
 \DC^1(X^2)\to \End(J)_\BQ
 $$
   are isomorphisms of $\BQ$-vector spaces.  Moreover, the last one is an isomorphism of $\BQ$-algebras.
   
    (2) For $z\in \DC^1(X\times A)$,  we have $\pr^t \circ z =z$.

      (3)  
    Under the isomorphisms in (1), the map 
 $$
\Hom^0(X,A)\times  \DC^1(X\times A)  \to \DC^1(X^2),\quad
 (f,z)\mapsto  \pr_f^t\circ z 
 $$
 is identified  with  
 the natural composition 
 $$
 \Hom(A^\vee, J)_\BQ \times  \Hom(J,A^\vee)_\BQ \to  \End(J)_\BQ.
 $$
 %Here ${\pr_f}\in \Ch^{\dim A}$ is the class of the graph of $f$, $\pr_f^t$ is the transpose correspondence.

  \end{prop}
 \begin{proof}
 (1) The first and third isomorphisms are special cases of \eqref{Sch3.9}.
 The second isomorphism follows from    \eqref{dualh} and \eqref{alb}  by definition.

(2) By definition, we have $(\pr^t \circ z)_{*}e=\pr^t_* \lb z_{*}e\rb=0$. 
By Corollary \ref{Deninger--Murre projectorscor} (2), we have $(\pr^t\circ z)^{*}u=z^\ast ( \pr_* u)=0.$
Thus $\pr^t \circ z\in \DC^1(X\times A)$.   By Corollary \ref{Deninger--Murre projectorscor} (1) (3), 
their images under the first isomorphism in   (1) are the same.  (2) follows.

(3) We may assume that $f\in \Hom(X,A)$.
 Then we have $ \pr_f^t \circ z= [\Gamma_f^t]\circ \pr^t \circ z=[\Gamma_f^t]\circ z$, where the last equation is by (2).
Now (3) follows from definition. 
 \end{proof}

% \begin{lem} \label{alg2}For $z\in \DC^1(X\times B)$,   $\pr^t \circ z =z$ %=z\circ \delta $.  \end{lem}

 % \begin{rmk}More generally, decomposing $\Ch^1(X\times B)$  as the direct sum of $\DC^1(X\times B)$and the pullbacks of $\Ch^1(X)$ and $\Ch^1(B)$  (as the analog of the combination of  \eqref{refm2}  and \eqref{refm}), the simultaneous left composition by $\pr^t$ and right composition by $\delta$ on  $\Ch^1(X\times B)$  is  the projection to $\DC^1(X\times B)$.  In particular,  for the motives, $\Hom\lb (X,\delta_X),(B,\pr^t)\rb_\BQ=\DC^1(X\times B)$.  %Similar statement holds for $\Hom\lb (X,\delta_X),({X^2},\delta_{X^2})\rb$ if ${X^2}$ is another curve.  Compare with Remark \ref{ted}. \end{rmk} 

We now apply the discussion in  \S\ref{Decomposition of the identity}  to $J=J_X$.  
 For $\phi\in \pi_A$, let $c_\phi\in  \DC^1(X\times A^\vee)$   (resp. $ f_{\phi}\in  \Hom^0(X,A ) $) be the corresponding element via  the first isomorphism in Proposition \ref{trans} (1) (resp. \eqref{alb}). 
 Taking dual homomophisms defines an anti-isomorphism $M_A\cong M_{A^\vee}$.
Identify $ \wt \pi_A$ and $ \pi_{A^\vee}$, as $M_A$-modules,  via \eqref{dualh} with $A$ replaced by $A^\vee$.
Then   the dual basis  $\{\wt \phi :\phi\in \cB_A\}\subset \wt \pi_A$ of $\cB_A$ 
becomes a basis of
 $ \pi_{A^\vee}$, and we denote it by $\cB_{A^\vee}$.
  \begin{prop}\label{vanlem2} 
 For the  $\pi_A$-isotypic  and  $\pi_{A^\vee}$-isotypic projectors $ \delta_{A},\delta_{{A^\vee}}\in\DC^1(X^2)$, we have
 \begin{equation}\label{5terms}
\delta_{A}  = \sum_{\phi\in \cB_A}\pr_{f_{\wt\phi}}^t\circ c_\phi
 =  \sum_{\phi\in \cB_A}c_{\wt\phi}^t\circ\pr_{f_{\phi}} ,
 \end{equation} 
 and
 \begin{equation}\label{5terms'}
\delta_{{A^\vee}}  = \sum_{\wt\phi\in \cB_{A^\vee}}\pr_{f_{\phi}}^t\circ c_{\wt\phi}
=  \sum_{\wt\phi\in \cB_{A^\vee}}c_{\phi}^t\circ\pr_{f_{\wt\phi}}   
.
 \end{equation}

 \end{prop}
  \begin{proof}
  Under $ \DC^1(X^2) \cong \End(J) $, taking transpose on $ \DC^1(X^2)$ corresponds to taking dual homomorphisms in $\End(J(X)) $. Then the remark under \eqref{decomproj} gives
\begin{equation}   \label{deltapid}
\delta_{A }^t= \delta_{{A^\vee}}.
\end{equation}
 Let $\delta_i$ be the  $i$th term in \eqref{5terms}  (so that $\delta_1=\delta_{A}$). 
Let $\delta_i'$ be the  $i$th term in \eqref{5terms'}. 
Then we have $\delta^t_2=\delta'_3$ and  $\delta'^t_2=\delta_3$.
By Corollary \ref{trans?}    and taking transpose, all $\delta_i$'s and $\delta_i'$'s  are in $\DC^1(X^2)$.
By   \eqref{decomproj} and  Proposition \ref{trans}  (3), we have $\delta_1=\delta_2$. Similarly, we have $\delta_1'=\delta_2'$. 
and hence $\delta_1=\delta'^t_1=\delta'^t_2=\delta_3$.
The proposition follows.
 \end{proof}

 \begin{cor}\label{vancor} For $i=1,...,n$, let
   $X_i$ be a  smooth projective  curve over $F$.   
Let $e_i\in \Ch^1(X_i) $ be  of degree one, and $ \delta_{X_i}$ defined as in \eqref{delta}.
  Let  $z\in   \Ch^d(X_1\times...\times X_n) $. Then 
 $\lb \delta_{X_1}\otimes...\otimes \delta_{X_n}\rb_* z=0$ holds if one of the following conditions  holds:

(1)  for  every $n$-tuple $(A_1,...,A_n) $, where $A_i$ is a   simple   isogeny  factor   $J_{X_i}$,   and   every
$f_i\in \Hom^0(X_i, A_i)_\BQ $,  we have
$ \lb  \pr_{f_1}\otimes...\otimes   \pr_{f_n}\rb_* z =0.$

(2)  for  every $n$-tuple $(A_1,...,A_n) $ as in (1) and
$c_i\in
\DC^1(X_i, A_i)$, we have $ \lb  {c_1}\otimes...\otimes   {c_n}\rb_* z =0.$

\end{cor}

   \section{The case of a Shimura curve}
   We use the results in the last section to decompose Hecke correspondences on Shimura curves.
Then    we study the automorphic structure of   Chow groups for a product of Shimura curves and prove the vanishing of certain
isotypic components of the diagonal cycle as predicted by the injectivity of the Abel--Jacobi map.  
Then we prove the theorems in the introduction.
    
    \subsection{Shimura curves and Hecke correspondences}\label{Shimura curves and Hecke correspondences}

  Let $F$ be a totally real number field, $\BA$  its ring of adeles and $\BA_\rf$  its ring of finite adeles.  A quaternion algebra $B $ over  $F$ is called {\em almost definite} if 
$B$ is  a matrix algebra at  one archimedean place $\tau:F\incl \BR$ and division at all other archimedean places of $E$.   For such an almost definite  quaternion algebra $B $,   and  an open compact subgroup $K$ of $B^\times(\BA_\rf)$, we have the smooth   compactified   Shimura curve $ X_K$ for  $B^\times$ of level $K$ over $F$ \cite[3.1.1]{YZZ}.   The  complex uniformization  of  $X_K$ via $\tau: F\incl \BR\subset \BC$ is given by
\begin{equation*}X_{K,\tau,\BC} \cong B^+ \bsl \BH \times B^\times(\BA_\rf)/K\coprod \{\text{cusps}\}.\end{equation*}
Here $B^+\subset B^\times$ is the subgroup of elements with totally positive norms (equivalently, with positive norms at $\tau$), $\BH$ is the complex upper half plane, and 
 the cusps exist if and only if $X_K $ is a modular curve, i.e. $F=\BQ$, and $B $ is the matrix algebra.  

   For $g\in B^\times(\BA_\rf)$, we have  the right multiplication isomorphism $\rho(g) :X_{gKg^{-1}}\cong X_K$.  Indeed, under  the above  complex uniformization  of  $X_K$, $\rho(g)$ acts by right multiplication on $B^\times(\BA_\rf)$. We also have      the natural projections $\pi_1:X_{K\cap gKg^{-1}}\to X_K$,    $\pi_2:X_{K\cap gKg^{-1}}\to X_{gKg^{-1}}$.
 Define the Hecke correspondence $Z(g)$ to be  the algebraic cycle on $X_K^2$ that is  the direct image of the fundamental   cycle   of $X_{K\cap gKg^{-1}}$ via $(\pi_1,\rho(g)\circ\pi_2)$. 
  % sufficiently small (the last condition is to remove the effect of the center).
In particular, $Z(1)$ is the diagonal of $X_K^2$.  
   % Let $$\HC(X^2)\subset \Ch^1(X^2)$$ be spanned by $[Z(g)]$'s as $g$ varies.

       The  Hodge  class  on $X_K$  for holomorphic modular forms of weight two is the        
        canonical   class     modified by the ramification points  \cite[3.1.3]{YZZ}. 
   (If  $K$ is neat,  it is the canonical   class        plus the sum of all cusps.) 
                 It is compatible under pullback and pushforward (up to degree) by  $\rho(g)$ and the natural morphisms of Shimura curves as the level  $K$ changes.  Let  $\vol(X_K)\in \BQ_{>0}$ be the degree of the Hodge  class on $X_K$.     
Then $$ d(g):=\deg \pi_1=\deg \pi_2=\frac{\vol\lb X_{K\cap gKg^{-1}}\rb}{ \vol\lb X_K\rb}.$$

    Let $e=e_K$ be the Hodge class  on $X_K$ divided by  $\vol(X_K)$  so that $e$ is of degree 1.  
      Then, we have 
 \begin{equation}\label{cHDZ}[Z(g)]_* e=[Z(g)]^* e=d(g)  e.
 \end{equation}
 Let 
 \begin{equation}\label{edelta}
 \wh e= e\times [X_K] +[X_K]\times e,\quad \delta=[Z(1)]-\wh e.
 \end{equation} 
   Then  $\delta$ coincides with  \eqref {delta} with $X=X_K$. 
And
 \eqref{cHDZ} implies that 
  \begin{equation}
  \label{cHDe}[Z(g)]\circ \wh e=d(g) \wh e.
 \end{equation} 
 
     Let $\cH_K$ be the convolution Hecke algebra of bi-$K$-invariant $\BC$-valued functions on $ B^\times(\BA_\rf)$ (with a given Haar measure on $ B^\times(\BA_\rf)$).
  Then $f\in \cH_K$  defines a   Hecke correspondence $$Z(f):=\sum_{g\in K\bsl B^\times(\BA_\rf)/K }\frac{\vol(KgK) }{d(g)}f(g) Z(g)\in \DC(X_K^2) _\BC  .$$
  The coefficients are used to make the following map    a $\BC$-algebra anti-homomorphism:
  \begin{equation}
  \label{HDC}\cH_K\to \DC(X_K^2)_{\BC},\quad f\mapsto [Z(f)].
  \end{equation}

 \subsection{Automorphic representations} \label{Automorphic representations}

For a simple abelian variety  $A$ over $F$, let 
 $$
\Pi_A=\leftidx{_B} \Pi_A:=\vil_K\Hom^0(X_K,A),
 $$
where $\Hom^0$ is as in Definition \ref{Hom0}, and we omit the left subscript untilt  \S \ref
{ss:thm2} where we use a  new quaternion algebra.
  The $\BQ$-vector space $\Pi_A$ carries a natural left action by $B^\times(\BA_\rf)$, induced by precomposing  the morphism $\rho(g) $ for $ g\in B^\times(\BA_\rf)$. (Indeed, $\rho $ induces a right action of $B^\times(\BA_\rf)$ on $\vpl X_K$.)
Then  the subspace of $K$-invariants is $\Pi_A^K=\Hom^0(X_K,A).$ 
 Let $M_A=\End(A)_\BQ$ so that  $M_A$ acts on $\Pi_A^K,\Pi_A$ from left by post-composition, and one may regard $\Pi_A$ as a $B^\times(\BA_\rf)$-representation with $M$-coefficients.

  \begin{thm}[{\cite[3.2.2, 3.2.3]{YZZ}}]\label{YZZ322} %by an isomoprhism \ol BQ_lBC
  Let $A$ be a simple abelian variety over $F$ such that 
 $
 \Pi_A\neq 0$. 
 
(1) The $\BQ$-algebra $M_A$ is a finite field extension of $\BQ$ and $[M_A:\BQ]=\dim A$.

(2)   Consider the set of pairs 
$ (A,\iota)$ where $A$ is a simple abelian variety over $F$ such that 
 $
 \Pi_A\neq 0$ and  $\iota: M_A\incl \BC$ is an embedding. The map $(A,\iota)\to \Pi_{A,\iota}$
 is a bijection to  the set of finite  components of 
   automorphic representations of $B^\times$, whose Jacquet--Langlands correspondence to $\GL_{2,F}$ is cuspidal and holomorphic of weight 2.

 \end{thm} 

Via \eqref{alb}, we have a right $\End(J(X_K))_\BC$-action on $\Pi_{A,\iota}^K$ by pre-composition.
Then the left $\cH_K$-module structure on $\Pi_{A,\iota}^K$ is also given by the composition  \begin{equation}
 \label{HCCE0}
 \cH_K\to\DC(X_K^2)_\BC\to  \DC^1(X_K^2)_\BC \cong \End(J(X_K))_\BC
 \end{equation} 
 and  the right action of $\End(J(X_K))_\BC$.
In \eqref{HCCE0}, the first map is the anti-homomorphism  \eqref{HDC},  and second is the projection via \eqref{refm2}, and the last one is the last  isomorphism of Proposition \ref{trans}   (1) (i.e. the composition of the last two is  the action of $\DC(X_K^2)$ on $J(X_K)$ by pushforward). 
  
  Note that $ \BC\wh e\subset \DC(X^2)_\BC$ is a subalgebra isomorphic to $\BC$ with multiplicative identity $\wh e$. Then $e$ span its unique simple module. 
Since   $\DC^1(X_K^2)$   annihilates $e$, $\DC^1(X_K^2)_\BC\cap \BC\wh e=0$ and the projection  
$\DC^1(X_K^2)\oplus \BC\wh e\to \BC\wh e$ is realized by the action on $\BC e$.

  \begin{prop} \label{Heckimage}
  (1) The image of  \eqref{HDC}   is $\DC^1(X_K^2)_\BC\oplus \BC\wh e$.   In particular, it is semisimple.

(2) A simple $\cH_K$-module  such  that the $\cH_K$-action factors through  via  \eqref{HDC}  is  
  some $\Pi_{A,\iota}^K$  as  in Theorem \ref{YZZ322}  or  $\mathbbm 1_K$.  Here  $\mathbbm 1_K$ is the representation of $\cH_K$ induced by the trivial representation of $B^\times(\BA_\rf)$.

    \end{prop}  
  \begin{proof}  Let  $\ol\cH_K$ be the image. By  Lemma \ref{zdelta} and  \eqref{cHDe},
  we have the following inclusion $$\ol\cH_K=\ol\cH_K\circ [Z(1)]=\ol\cH_K\circ(\delta+\wh e)\incl \DC^1(X_K^2)_\BC\oplus \BC\wh e.$$
 By  \eqref{cHDZ}  and  the discussion between Theorem \ref{YZZ322}  and the proposition, this inclusion map is induced by  the action of $\cH_K$ on 
    $\Pi_A^K\otimes_{\BQ} \BC$'s  as $A$ varies, and  $\mathbbm 1_K$. 
   By Theorem \ref{YZZ322} (2), these simple $\cH_K$-modules   are pairwise non-isomorphic.
 By    the density theorem of   Jacobson and Chevalley, 
the first part  of (1) follows. The second part of  (1) follows from the last  isomorphism of 
 \eqref{HCCE0} and the discussion in \S\ref{Decomposition of the identity}.
 
 The proof of  (2) is included in the proof of the first part  of  (1).
         \end{proof}
         \begin{rmk}(1) One may deduce the proposition using cohomology instead of Jacobian. Indeed  Theorem \ref{YZZ322} is proved  in \cite{YZZ} in this way. We use Jacobian to make the discussion closer to \S \ref{An alternative approach}.
         
         (2) The proposition in particular implies that $\wh e$ is a linear combination of Hecke correspondences.
  \end{rmk}
       
    For $\pi$ being one of simple $\cH_K$-module as in Proposition \ref{Heckimage} (2), let 
  \begin{equation}
 \label{proj1}\delta_\pi\in \DC^1(X_K^2)_\BC\oplus \BC\wh e
  \end{equation} 
be the $\pi$-isotypic  projector. 
 Then clearly 
  \begin{equation}
 \label{proj2}
 \delta_{\mathbbm 1_K}=\wh e.
  \end{equation}

 \begin{rmk} \label{cusps}By Theorem \ref{YZZ322},   the $\mathbbm 1_K$-component of $\Ch^1(X_K)_\BC$ is if dimension 1, and thus spanned by $   e$.  When $X_K$ is a modular curve, the divisor class $\fc$ of the sum of all cusps lies in the $\mathbbm 1_K$-component. So it is a $\BQ$-multiple of $e$, and thus both $\fc,e$ are $\BQ$-multiples of the canonical class if  $K$ is neat.
Moreover, by 
  the   theorem of Manin  \cite{Man} and Drinfeld \cite{Dri} on the torsionness of degree 0 cuspidal divisors on modular curves, all of them are equal to any single cusp.
 % (both are $\BQ$-divisor classes).
 \end{rmk}
 The other projectors  are given by the base-change-to-$\BC$ version of
 Proposition \ref{vanlem2}.

    \subsection{Product of Shimura curves}  \label{Product of Shimura curves}  
   
 For $i=1,...,n$, let $X_i$ be the Shimura curve for an almost definite quaternion algebra $B_i$ over $F$ of level $K_i\subset B_i^\times(\BA_\rf)$.  Let $\cH=\otimes_{i=1}^n\cH_{K_i}.$   By   Proposition \ref{Heckimage} and \cite[(2.5)]{Lam}, 
  an $\cH$-module  $\cM$ such  that the $\cH$-action factors through  the tensor  of  the anti-homomorphisms  \eqref{HDC}  for all $X_i$'s
 \begin{equation}\label{HDCtensor}
 \cH\to\mathop\otimes_{i=1}^n\DC(X_i^2)_{\BC}
 \end{equation}
   is a direct sum of  simple modules of the form
  $$\pi=\mathop\boxtimes\limits _{i=1}^n \pi_i.$$
  Here   $ \pi_i=\mathbbm 1_{K_i}$ or $\pi_i=\Pi_{A_i,\iota_i}^{K_i}$  
where $ A_i$ is a    simple abelian variety over $F$ such that  $\Pi_{A_i}^{K_i} \neq 0,$ and $\iota_i:M_{A_i}\incl \BC$  is an embedding.    
In particular, $\pi
 $ is automorphic and  it is
  cuspidal if none of the $\pi_i$'s is $\mathbbm 1_{K_i}$ (see Theorem \ref{YZZ322}).
  The $\pi$-isotypic projector   is the tensor of the ones  for each $\pi_i$ (see  \eqref{proj1}):   
$$\delta_{\pi} =\mathop\otimes_{i=1}^n \delta_{\pi_i} \in \mathop\otimes_{i=1}^n\DC^1(X_i^2)_\BC.$$  
  For our purpose, let $  \cM=\Ch^*(X_{1...n})_\BC$, where $X_{1...n}$ is the product of $X_i$'s and $ \otimes_{i=1}^n\DC(X_i^2)_{\BC}
$ acts on $  \Ch^*(X_{1...n,E})_\BC$ by pullback  so that $\cM$ is a left $\cH$-module. 
We may also take arbitrary  field extension of $F$.

  For $\pi$ being cuspidal,   by the obvious base-change-to-$\BC$ version of
 Proposition \ref{vanlem2},
  we have the following analog of  Corollary \ref{vancor} (1) (and we leave the analog of (2) to the reader).

 \begin{prop}
 \label{refinecr} Assume that $\pi$ is cuspidal with $\pi_i=\Pi_{A_i,\iota_i}^{K_i}$  for $i=1,...,n$. 
      For $z\in  \Ch^*(X_{1...n,E})_{\BC}$,  
 $\delta_{\pi,*}z   =0$ if    $  \lb \otimes_{i=1}^n  \pr_{f_i}\rb_* z =0$ for all $f _i \in  \pi_i $. Here $\pr_{f_i}$ is defined in Definition \ref{prf}.\end{prop}

 Now we fix an almost definite quaternion algebra $B$ over $F$, and consider  the $n$-th power  of  the Shimura curve $ X_K$, $K\subset B^\times(\BA_\rf)$.
 Let  $$\Ch(n)=\vpl_K  \Ch^*(X_{K}^n) _\BC,$$
% \WZ{Are we really  using the limit over $E/F$ in the later application below?}
 where  the transition maps in the inverse limit are pushforwrds   by the natural morphisms between Shimura curves. 
Let $ B^{\times,n}(\BA_\rf)$ act on $\ol\Ch(n)$ (from left) by   pullback.
  Let $\Pi=\boxtimes_{i=1}^n\Pi_{i}$ where  $\Pi_i$  is the trivial representation   of $B^\times(\BA_\rf)$
 or   of the form $\Pi_{A,\iota}$, as a representation of  $ B^{\times,n}(\BA_\rf)$. Let  $\Pi^K=\boxtimes_{i=1}^n\Pi_{i}^{K}$.   
 \begin{thm}\label{general}
 Let $ H\subset  B^{\times,n}(\BA_\rf)$ be a subgroup such that
$\Pi$
   has no nonzero  $H$-invariant linear forms. Then for $z=(z_K)_K\in\Ch(n) $ invariant by $H$, we have $ \delta_{\Pi^{K},*}z_K =0$ for all $K$.
\end{thm}
\begin{proof} 

First, assume that $\Pi_n$ is the trivial representation. Let $z'_K$ be the pushforward  of $z_K$  by the projection to the product of  first  $(n-1)$ copies of $X_K$.  Let $H'  $ be the projection image of $H$ to 
 the product of  first  $(n-1)$ copies of $B^\times(\BA_\rf)$.
By \eqref{proj2}, we have $$    \delta_{\Pi_n^K,*}z_K =e_K\times z'_K+[X_K]\times z_{K,*}e_K ,$$
which is 0 if and only if 
$z'_K=0$ and $   z_{K,*}e_K=0$. 
It is easy to check that
$ (z'_K)_K, \lb z_{K,*}e_K\rb_K \in \ol\Ch(n-1) $, and they are $H'$-invariant. 
 Moreover, $\Pi$ 
   has no nonzero $H$-invariant linear form if and only if  $ \otimes_{i=1}^{n-1} \Pi_{i}$
    has no nonzero $H'$-invariant linear form.
 Thus  we only need to consider the case that  $\Pi_i=\Pi_{A_i ,\iota_i } $  for all $i=1,...,n$. 
 Now we assume this is the case.
  
  Consider  the     $\BC$-linear (the linearity on the first component follows from Definition \ref{prf}) map 
\begin{align*}  
  \boxtimes_{i=1}^n\Pi_{A_i,\iota_i}  
\times  \Ch(n)&\to \Ch_{1}(A_1\times...\times A_n) _{\BC},\\
  \lb \mathop\otimes_{i=1}^n  {f_i} , x \rb &\mapsto \lb \mathop\otimes_{i=1}^n  \pr_{f_i}\rb_* x  .
  \end{align*}
It is direct to check its
  $  B^{\times,n}(\BA_\rf)$-equivariance, i.e., for  $g=(g_1,...,g_n)\in B^{\times,n}(\BA_\rf)$
 $$   \lb \mathop\otimes_{i=1}^n g\cdot f \rb_* ( g^* x)= \lb \mathop\otimes_{i=1}^n  \pr_{f_i}\rb_* x .
 $$  
Then, evaluating it  at an $H$-invariant element $z\in  \Ch(n)$, we obtain an  $H$-invariant  $\BC$-linear  map 
\begin{equation}\label{YZZmap} \begin{split} \Pi=\mathop\boxtimes_{i=1}^n\Pi_{A_i,\iota_i} &\to  \Ch_{1}(A_1\times...\times A_n) _{\BC},
\\
 \mathop\otimes_{i=1}^n  {f_i} &\mapsto \lb \mathop\otimes_{i=1}^n  \pr_{f_i}\rb_* z  .
 \end{split}
 \end{equation}

The assumption on $\Pi$ forces this map to be 0.
 The theorem  follows from Proposition  \ref{refinecr}. 
\end{proof}
  \begin{rmk} \label{BBrmk}
 (1)  The map \eqref{YZZmap}  appeared in \cite{YZZ0}  for the diagonal 1-cycle, see the corollaries below.

%(2) Concerning the  left $\cH$-module structure on $  \Ch^*(X_{K,E}^n)$ as above Proposition \ref{addchow}, \eqref{deltapid} shows that $    \delta_{\Pi^K,*} z   $ is  the $\wt\Pi^K$-isotypic component of $z$.  Here $\wt\Pi$ is the  admissible dual of $\Pi$.

 %One may also apply the analog of criterion (2) in Proposition  \ref{refinecr} to prove Theorem \ref{general} in the same manner. 

 (2)      Theorem \ref{general}  is predicted by  the  conjectural injectivity of the Abel--Jacobi map.
   
 (3) By the proof of  Theorem  \ref{general},  only the cuspidal case is of real interest.  Note that  cuspidal representations $B^\times$ only appears in the $H^1$ of $X_K$.
 If $z$ is an 1-cycle,  then  the  conjectural injectivity of the Abel--Jacobi map shows that only the case $n=3$  is of real interest.   

\end{rmk}
 
\begin{cor}\label{generalcor}
 Let $\Delta_K $ be the diagonal embedding of $X_K$ in $ X_K^n$.
 Assume that  $\Pi$
   has no nonzero diagonal-$B^\times(\BA_\rf)$-invariant linear forms, then $    \delta_{\Pi^K,*} [\Delta_K]  =0.$
\end{cor}
\begin{proof}  In Theorem  \ref{general}, let $z_K= \frac{1}{\vol(X_K)} [\Delta_K] . $ Then $z=(z_K)_K\in\Ch(n) $   is obviously invariant under $H$, the diagonal embeding of $B^\times(\BA_\rf)$. The assertion follows.
\end{proof}

\subsection{Applications of Prasad's theorem on trilinear forms}
\label{Triple product diagonal cycle}
In this subsection, we prove Theorem  \ref
{thm:3} and \ref
{thm:4} in the introduction. 

Before that, we   recall Prasad's theorem  \cite{Pra,Pra1} about trilinear forms and  local root numbers for later use.
The reader may skip the next  lemma and theorem, and come back for  reference.
Let $k$ be a non-archimedean local field and $D$ the unique division quaternion algebra over $k$.
The following lemma is obvious.
\begin{lem}\label{lem:cc}Let $V_1,V_2,V_3$  be   irreducible    representations of $G=\GL_2(k),D^\times$.
  If   the product of the  central  characters  of  $V_1,V_2,V_3$ is not trivial, then  
   $$ \Hom_{G}\lb V_1\otimes V_2\otimes V_3,\BC\rb=0.$$

    \end{lem}
  
   The main results of  \cite{Pra,Pra1} are as follows

 \begin{thm}\label{thm:Pra}
 Let $V_1,V_2,V_3$  be irreducible  infinite-dimensional
   representation of $ \GL_2(k)$ such that
 the product of their central  characters is trivial. Then
 $$ \dim \Hom_{\GL_2(k)}\lb V_1\otimes V_2\otimes V_3,\BC\rb+  \dim \Hom_{D^\times}\lb \JL(V_1)\otimes \JL(V_2)\otimes \JL(V_3),\BC\rb=1.$$
 Here $\JL$ denote Jacquet--Langlands correspondence  from $ \GL_2(k)$  to $D^\times$, and 
 if $V_i$ is not a  discrete series, $\JL(V_1)$ is understood as $0$.
Moreover,   the triple product local root number  $\ep:=  \ep \lb V_1\otimes V_2\otimes V_3\rb$    is $\pm 1$.
And  $$ \dim \Hom_{\GL_2(k)}\lb V_1\otimes V_2\otimes V_3,\BC\rb=1\text{ if and only if }\ep=1.$$

\end{thm}

Now we start to prove our theorems.
Continues to use the   notations in the last subsection, the totally real field $F$, quaternion algebra $B/F$ and Shimura curve $X_K$ which we further abbreviate to $X$ in this subsection. 
Assume that $n=3$ and let $\Delta $ be the diagonal embedding of $X $ in $ X^3$.

 Let us first prove Theorem \ref
{thm:4}.

\begin{proof}
 [Proof of Theorem \ref
{thm:4}]

Assume that  $[\Delta]_{\otimes_{i=1}^3 h^{E_i}(X) }\neq 0$ for our $X$. 
Then $ \Pi_{E_i}^K\neq 0$ and  $\delta_{\Pi_{E_i}^K}=\delta_{E_i}\in \DC^1(X^2)_\BC$. 
Thus 
$      \delta_{\Pi^K,*}[\Delta]  \neq 0$ where $ \Pi=\Pi_{E_1,\BC}\otimes\Pi_{E_2,\BC}\otimes\Pi_{E_3,\BC}$.  
By Corollary \ref{generalcor},  $\Pi$ has   nonzero diagonal-$B^\times(\BA_\rf)$-invariant linear forms. By
Prasad's dichotomy   (Theorem  \ref{thm:Pra} (1)), such quaternion algebra is unique. 
  \end{proof}
More generally, let $A_i,i=1,2,3$ be  simple abelian varieties    over $F$. 
  Recall $M_{A_i}=\End(A_i)_\BQ$ acts on $\Pi_{A_i}$. 
 Theorem \ref
{thm:4} also holds for   $A_i,i=1,2,3$   after replacing ``at most one" by ``at most 
$[M_{A_1}:\BQ] [M_{A_2}:\BQ] [M_{A_3}:\BQ] $". The proof is the same after incorporating  the following lemma, and we leave the proof to the reader.

\begin{lem} \label{lem:AvsPi}The isotypic component $[\Delta]_{\otimes_{i=1}^3 h^{A_i}(X) }= 0$   if 
 $      \delta_{\Pi^K,*}[\Delta] = 0$, where 
  $\Pi=\otimes_{i=1}^3\Pi_{A_i,\iota_i}$, for every triple $\lb \iota_i:M_{A_i}\incl \BC,\, i=1,2,3\rb$.

\end{lem}
\begin{proof} By definition,  $\delta_{A_i}=\sum\limits_{\iota:M_{A_i}\incl \BC} \delta_{\Pi^K_{A_i,\iota}}\in \DC^1(X^2)_\BC$. The lemma follows.\end{proof}

Before we proceed, we need some preparations.

For a quaternion algebra $D$ over a number field $E$ and  a place $v$ of $E$,  the Hasse invariant $\ep(D_v)$ is 1 if $D_v$  is a matrix algebra, $-1$ if $D_v$ is division. The following theorem is  well known.
\begin{thm} [Hasse principle  for quaternion algebras]

 (1) The product over all places  $\prod_v \ep(D_v)=1$.
 
 (2) Given $\ep_v\in \{\pm 1\}$ for all places $v$ of $E$ such that $\ep_v=1$  for all  but finitely many places of $E$ and  $\prod_v \ep_v=1$.
 Then there is a quaternion algebra $D$ over a number field $E$ such that $\ep(D_v)=\ep_v$ for all $v$.
 \end{thm}

 Now we come back to   the  simple abelian varieties  $A_i,i=1,2,3$   over the totally real field $F$. 
 Let    $\Pi=\otimes_{i=1}^3\Pi_{A_i,\iota_i}$, for some triple $\lb \iota_i:M_{A_i}\incl \BC,\, i=1,2,3\rb$.
 Let $\ep(\Pi_v)$ be the local root number of the Jacquet--Langlands correspondence of $\Pi_v$.  
\begin{prop}\label{general3}  
  If  $\ep(\Pi_v)\ep(B_v)\neq 1$ for some finite place $v$, then 
$      \delta_{\Pi^K,*}[\Delta]  =0$. 
\end{prop}
\begin{proof}
If the restriction of the central character of $\Pi_v$ to the diagonal $F_v^\times$ is non-trivial, 
  the assumption in  Corollary \ref {generalcor} holds by Lemma \ref{lem:cc}.  The proposition follows.
Otherwise,  we must have $\ep(\Pi_v)\in \{\pm1\}$. If  $\ep(\Pi_v)\ep(B_v)=- 1$ for some finite place $v$, by    Prasad's theorem  \cite{Pra,Pra1} relating trilinear forms and  local root numbers, we have $\Hom_{H}(\Pi_v,\BC)=0$, where $H$ is the   diagonal$B_v^\times$.   The corollary now follows from   Corollary \ref {generalcor}. 
 \end{proof} 
  \begin{rmk} 
  
Proposition \ref{general3}
  can be used to pin down the unique (if it exists) quaternion in Theorem \ref
{thm:4}.
  \end{rmk}
  
Before we prove Theorem \ref
{thm:3},  we need the holomorphy of the triple product $L$-function. 
  \begin{thm} [\cite{Gar,PSR}] \label{PSR}For three cuspidal automorphic representation $\sigma_i,i=1,2,3$ of $\GL_{2,F}$,
  the L-function $L\lb s,\sigma_1\otimes \sigma_2\otimes \sigma_3\rb$ has holomorphic continuation to $s\in\BC$.  
  \end{thm}

  \begin{cor}  \label{holom}The L-function $L\lb s,h^1(A_1)\otimes h^1(A_2)\otimes h^1(A_3)\rb$ has holomorphic continuation to $s\in\BC$.  
  \end{cor}

\begin{proof}
 [Proof of Theorem \ref
{thm:3}]    
We may take  $X$ and $A_i$'s  as the ones we use above.
 Assume $[\Delta]_{h^{A_1}(X)\otimes h^{A_2}(X)\otimes h^{A_3}(X)}\neq 0,$  
we want to prove that $$L\lb 2,h^1(A_1)\otimes h^1(A_2)\otimes h^1(A_3)\rb =0. $$

By Lemma \ref{lem:AvsPi},
 for some  $\Pi=\otimes_{i=1}^3\Pi_{A_i,\iota_i}$, $      \delta_{\Pi^K,*}[\Delta]  \neq0$.  
 Let $\pi$ be the corresponding triple   product of  cuspidal automorphic representations of $\GL_{2,F}$  that are  holomorphic of weight 2 via the Jacquet--Langlands correspondence, as in Theorem \ref{YZZ322} (2). 
   By  Corollary \ref {generalcor},
  the restriction of the central character of $\pi_v$ to the diagonal $F_v^\times$ is  trivial for every finite place $v$ of $F$. By Proposition \ref{general3}, $\ep(\pi_v)\ep(B_v)=1$.
   At  every infinite place $v$,   $\ep(\pi_v)=-1$ (see \cite[Section 9]{Pra}).    At  every infinite place $v\neq \tau$, $\ep(B_v)=-1$ and $\ep(B_\tau)=1$. 
 Thus by the Hasse principle,
  $\ep(\pi)=\prod_v \ep(\pi_v)=-1$. So $L(1/2,\pi)=0$ and thus    $L\lb 2,h^1(A_1)\otimes h^1(A_2)\otimes h^1(A_3)\rb =0$ as it contains   $L(1/2,\pi)$ as a multiplicative factor (while the other factors are finite by Theorem \ref{PSR}).
This is a contradiction.
 \end{proof}

 \subsection{Proof of Theorem \ref{thm:2}}\label{ss:thm2}
Below, we prove Theorem \ref{thm:2} by proving  a  more general  result (Corollary \ref{cor:2+}).
However, for now we put the non-hyperellipticity aside,  and focus on the vanishing of components of diagonal cycles.

\begin{thm}
 \label
{thm:2+}
Let $B$ be an almost definite    quaternion algebra  over a totally real field $F$.
Let $A_i,i=1,2,3$ be  simple abelian varieties    over $F$ such that $\leftidx{_B}\Pi_{A_i}\neq 0$. 
 Assume that    for some finite place $v$ of $F$, $\leftidx{_B}\Pi_{A_i,v}$    is a discrete series for $i=1,2,3$. 
  Then there are a finite totally  real   extension $E/F$ and an almost definite    quaternion algebra $D $ over $E$ 
  such that  \begin{itemize}
  \item[(1)]  the base change   abelian variety $A_{i,E}$ is simple,
  \item[(2)]   we have $\leftidx{_D}\Pi _{A_{i,E}}\neq 0$,% where for a simple abelian variety $A'$ over $E$, we  write
 %$$
%\leftidx{_D}\Pi _{A' }:=\vil_{L\subset D^\times(\BA_{E,\rf})}\Hom^0(X_L,A')
% $$
% to indicate the dependence on the new quaternion algebra $D$,
 \item[(3)] 
 for any open compact subgroup  $L\subset D^\times(\BA_{E,\rf})$,  the diagonal embedding $\Delta_L$ of $X_L$ in $X_L^3$ satisfies 
$$[\Delta_L]_{{h^{A_{1,E}}(X_L)\otimes h^{A_{2,E}}(X_L)\otimes h^{A_{3,E}}(X_L)}}= 0.$$
 \end{itemize}

  \end{thm}

We need  some preparations. 
\begin{lem}\label{lem:successive quad extensions } 
 
  Let  $v_1, v_2,...,v_m$ be   finite places of a totally real field  $F$. There exists a totally real quadratic extension $E/F$ split at $v_1,...,v_m$.

\end{lem}
\begin{proof}

Let $F_\infty=\prod_{u|\infty} F_u^\times$. Let $F_\infty^+\subset F_\infty^\times$ be the subgroup of totally positive elements. 
 In $\prod F_u^\times$ where $u$ runs over all infinite places  of $F$, $v_1, v_2,...,v_m$  and  $v$, embed $F^\times$ diagonally. By weak approximation, in the product $F_\infty^\times \times \prod_{i=1}^m F_{v_i}^{\times}\times F_v^\times$ with $F^\times$ embedded diagonally,  we have 
 $$F^\times\cap \lb F_\infty^+\times \prod_{i=1}^m F_{v_i}^{\times,2}\times (F_v^\times-F_{v}^{\times,2})\rb \neq \emptyset.$$
Let $E=F(\sqrt d)$ for some $d$ in this set. 
 \end{proof} 
 The following lemma is obvious by  the Hasse principle.
 \begin{lem}\label{lem:nonhyp} 
 
  Let  $v_1, v_2,...,v_{m}$ be different finite places of  a totally real field $E$,
  and $\ep_k\in \{\pm1\}$ for $k=1,...,n<m$.  There is an    almost definite   quaternion algebra  $D$ over $E$ such that 
 \begin{itemize}
   \item $D$ is  a matrix algebra at every finite place outside $\{v_1,v_2,...,v_{m}\}$, 
\item  $\ep \lb D_{v_k}\rb\neq\ep_k$ for $k=1,...,n $.
\end{itemize}

\end{lem}
 
 The choice of $E$ in the theorem will be given by the following proposition.
\begin{prop}\label{ChoseE}
 Fix a totally real field $F$, 
 a finite place $v$ of $F$,  a positive integer $ n$ and an almost definite quaternion algebra $B$ over  $F$. Then
there exists  a finite totally  real   extension $E/F$ with 
places $ v_1,...,v_{m},m>n,$  of $E$ of degree one over $v$ (i.e., $E_{u}\cong F_v$ for $u\in \{v_1,...,v_{m}\}$),
such that
\begin{enumerate} 
\item for every   simple abelian variety $A$ over  $F$ with $\Pi_A:=\leftidx{_B}\Pi_A\neq 0$ and $ \Pi_{A,v}$   being a discrete series,  we have that
$A_E$ is simple and the natural map $M_A\to M_{A_E}$ is an isomorphism,
 \item
  for every  almost definite  quaternion algebra $D $ over $E$  
 that is  a matrix algebra  locally at every finite place outside $\{v_1,v_2,...,v_{m}\}$,   
 we have that $ \Pi_{A_E}:=\leftidx{_D}\Pi_{A_E}\neq 0$ and  $\Pi_{A_E,\iota,u}$ is the Jacquet--Langlands correspondence of  $\Pi_{A,\iota,v}$   under
$E_{u}\cong F_v$ for all $\iota:M_A\cong M_{A_E}\incl \BC$ and  $u\in \{v_1,...,v_{m}\}$.
%\WZ{I donot know what is the conclusion of this item: ``for..., such that..." then what?? Unclear to me even after reading the proof. It seems you now want to fix the places   $ v_1,...,v_{m},$ in order to state (3). Anyway, this Prop. is still a mess to me. }
%\WZ{DHere if $D_u\cong B_v$, the Jacquet--Langlands correspondence is understood as the identity. }
 \end{enumerate}
\end{prop}
 
\begin{proof}

Let  $ m=2^r>n$ be a power of $2$. Let $E_1/F$ be a totally real quadratic extension   split at $v$, and $E_2/E_1$ 
 a totally real quadratic extension $E/F$ split at places over $v$ and so on until we get $E:=E_r$ such that $v$ is completely split along the extension $E/F$. The existence of $E_1,...,E_r$ is guaranteed by Lemma \ref{lem:successive quad extensions }. 
Let $v_1,...,v_{m}$ be all the places of $E$ over $F$.

To verify (1)(2), the key is the  quadratic base change by Langlands \cite{Lan}.
For $\Pi_{A,\iota}$  where $\iota:M_A\cong M_{A_E}\incl \BC$, 
let $\pi$ be the  corresponding cuspidal automorphic representation of $\GL_{2,F}$     holomorphic of weight 2 via the Jacquet--Langlands correspondence to $\GL_2$, as in Theorem \ref{YZZ322} (2).
The successive  base change  $\pi_E$ of $\pi$ to $E$ via the \textit{totally real} quadratic tower $E=E_r/E_{r-1}/.../E_1/F$ exists and is still cuspidal automorphic \cite{Lan}. 
Consider the Hecke field  $H(\pi)$, i.e., the field over $\BQ$ generated by the unramified Hecke eigenvalues of $\pi.$. Then by definition $H(\pi_E)\subset H(\pi)$.
 It is well known to experts that  in fact 
 \begin{equation}
\label{Hpp}H(\pi_E)=H(\pi).
\end{equation}
  (This essentially follows from the density for places of $F$ split in $E$.
  For the lack of a proper reference, we sketch a proof. 
Let $S$ be the set of finite places of $F$ where $\pi$ or $E$ is ramified. 
Consider the 2-dimensional $ \lambda$-adic representation $\rho$ of $\Gal(F^S/F)$
 associated  $\pi$, where $ \lambda$ is the finite place of $H(\pi)$ over an inert prime $\ell$ of $\BQ$ and $F^S\subset \ol F$ is the maximal unramified outside $S$ extension of $F$. Then the Frobenious traces  of $\rho$ outside $S$ are the same as Hecke eigenvalues  of $\pi$.
  By 
 Chebotarev's density theorem, $H(\pi)_\lambda$ is generated over $\BQ_\ell$ by the Frobenious traces at finite places of $F$ outside $S$ that are split in $E$.  The Frobenious traces are the same for the corresponding places in $E$ under $H(\pi_E)\subset H(\pi)$. Thus $H(\pi_E)\BQ_\ell=H(\pi)_\lambda$ and so $H(\pi_E)=H(\pi).$)

Now we can we verify (1)(2). 
For $u\in \{v_1,...,v_{m}\}$, since $\pi_{E,u}\cong \pi_v$ under $E_{u}\cong F_v$,
  $\pi_{E,u}$ is a discrete series. Since  $D$ is 
   a matrix algebra   at every finite place outside $\{v_1,v_2,...,v_{m}\}$,  
 the Jacquet--Langlands correspondence $ \pi'$ of  $\pi_E$   to  $D^\times$   is nonzero. By 
Theorem \ref{YZZ322} (2), $ \pi'_{\rf}\cong \leftidx{_D}\Pi_{A',\iota}$  for some 
simple abelian variety $A'$ over $E$ and an embedding 
  $\iota: M_{A'}\incl \BC$.  
  By \cite[3.2.3]{YZZ}, $H(\pi)\cong M_{A}, H(\pi_E)\cong M_{A'}$. 
Then  by \eqref{Hpp}, $M_{A} \cong M_{A'}$. Recall that by Theorem \ref
{YZZ322},  $[M_{A_i}:\BQ]=\dim A_i$. 
  Thus $$\dim A'=[M_{A'}:\BQ]= [M_{A}:\BQ]  =\dim A=\dim A_E.$$
  As both the   Galois representations associated to $A'$ and $A_E$  contain a direct summand corresponding to $\pi_E$ under the Langlands correspondence, by the isogeny theorem of Faltings, $\Hom(A',A_E)\neq 0$. Since $A'$ is simple and of the same dimension with $A_E$, $A_E\cong A'$ up to isogeny and is simple. 
  So $M_{A_E} \cong M_{A'}\cong M_A$ and  (1)  is proved.
  Also $\leftidx{_D}\Pi_{A_E}\cong  \leftidx{_D}\Pi_{A'}\neq 0$. Then the  final Jacquet--Langlands correspondence  follows by definition.
%So  understanding  $\iota:M_A\cong M_{A_E}\cong M_{A'} \incl \BC$, for    $u\in \{v_1,...,v_{m}\}$,
 %$\Pi_{A_E,\iota,u}
 % $ is the Jacquet--Langlands correspondence of $\pi_{E,u}\cong \pi_v$ to  $D_u^\times $ under
%$E_{u}\cong F_v$, which is the  Jacquet--Langlands correspondence of $ \Pi_{A,\iota,v}$. 
(2) is proved. 
      \end{proof} 

\begin{proof}[Proof of Theorem \ref{thm:2+}]

 If the restriction of the central character of $\Pi$ to the diagonal $\BA_\rf^\times$ is non-trivial, 
  the assumption in  Corollary \ref {generalcor} holds by Lemma \ref{lem:cc}.  The theorem follows with $E=F,D=B$.
Below, assume the restriction of the central character of $\Pi$ to the diagonal $\BA_\rf^\times$ is  trivial, 

First, let $J$ be the set of triples $\lb \iota_i:M_{A_i}\incl \BC,\, i=1,2,3\rb$ and $n=|J|$. (Then $n=[M_{A_1}:\BQ] [M_{A_2}:\BQ] [M_{A_3}:\BQ]$.) 
Let 
 $E$  and $ \{v_1,...,v_{m}\},\, m>n$ be as in Proposition \ref {ChoseE} so that (1) of the theorem holds.
 Moreover,  $\iota_i:M_{A_i}\incl \BC$ is also understood as an embedding $M_{A_{i,E}}\incl \BC$ via Proposition \ref {ChoseE} (1), and any latter embedding  comes in this way.

Second, we   construct $D$ as follows.
 Give $J$ an arbitrary order. 
For $ k=1,...,n$, if the $k$-th element of $J$ is $(\iota_1,\iota_2,\iota_3)$, let 
$\Pi_k=
\otimes_{i=1}^3\Pi_{A_i,\iota_{i}}$. 
Let $\ep_k=\ep(\Pi_{k,v})$. Apply Lemma \ref{lem:nonhyp} to get $D$.  Then by the  first condition in Lemma \ref{lem:nonhyp} and Proposition \ref {ChoseE} (2), (2) of the theorem holds.

 Third,  we verify the equation in (3) of  the theorem.  Let $(\iota_1,\iota_2,\iota_3)$ be the $k$-th element of $J$.
 By Proposition \ref {ChoseE} (2),
$\lb \otimes_{i=1}^3\Pi_{A_{i,E},\iota_{i }}\rb_u$ is the (triple product)  Jacquet--Langlands correspondence of $\Pi_{k,v}$ for $u\in \{v_1,...,v_{m}\}$. 
Since $\ep(D_{v_k})\neq\ep_k=\ep(\Pi_{k,v})$ (this is the second  condition in Lemma \ref{lem:nonhyp}), we
let $u=v_k$. Then 
 Proposition \ref{general3}  (with $B,\Pi,v$ there replaced by $D, \otimes_{i=1}^3\Pi_{A_{i,E},\iota_{i }}, u=v_k$ here) implies that $$\delta_{\lb \otimes_{i=1}^3\Pi_{A_{i,E},\iota_{i }}\rb^L,*}[\Delta_L]=0.$$
 As this holds for all triples $(\iota_1,\iota_2,\iota_3)$, by Lemma \ref{lem:AvsPi}, we have   
 $$[\Delta_L]_{{h^{A_{1,E}}(X_L)\otimes h^{A_{2,E}}(X_L)\otimes h^{A_{3,E}}(X_L)}}= 0.$$
   \end{proof}

Now we deal with non-hyperellipticity.  We need the following result of independent interest.
 
\begin{prop}\label{prop:nonhyp} Let $A$ be a simple abelian variety over a totally real field $F$.
 If $\Pi_A:=\leftidx{_B}\Pi_A\neq 0$ for an almost definite quaternion algebra $B$ over  $F$,  then for any $d>0$, there is   some $K\subset B^\times(\BA_\rf)$ such that   connected components of $X_{K,\BC}$ are non-hyperelliptic and $\dim_\BQ \Pi_A^K>d|\pi_0 (X_{K,\BC}) |$.

   \end{prop}
\begin{proof}  
      For an ideal $\fn$ of $\cO_F$ coprime to all places where $B$ is division, let $ X_0(\fn)$   be the $\Gamma_0(\fn)$-level Shimura curve, i.e., $X_0(\fn)=X_K$  where   $K\subset B^\times(\BA_\rf)$ is  the closure of the so-called Eichler order of level $\fn$. 
      A direct computation shows that $  \pi_0\lb X_0(\fn)_\BC\rb=\Cl_F^+$, the narrow ideal class group. In particular, 
      $ | \pi_0\lb X_0(\fn)_\BC\rb|$ is a constant as $\fn$ changes. 
   By the standard theory of admissible representations of $\GL_2$ over  local field, we can choose (infinitely many) $\fn$   so that 
   $\dim \Pi_A^K>d|\pi_0 (X_{K,\BC}) |$. 
   Finally, we only need to show that there are only finitely many $\fn$ such that   connected components of $X_{K,\BC}$ are  hyperelliptic. This can be proved using  Ogg's method  in  \cite[\S5]{Ogg1}, where  $F=\BQ$ was considered.  When $B$ is a matrix algebra over $\BQ$, Ogg  used supersingular points on the smooth integral model of $Y_0(\fn)$ over $\cO_{F,\fp}$ for some prime $\fp\nmid \fn$. The supersingular points are defined over the unique separable quadratic extension  of the residue field at $\fp$.  For general $F$, we can use the method of Carayol \cite{Car} to define an integral model and use the basic loci (the generalization of the set of supersingular points) to run the proof of Ogg. 
\end{proof}  
\begin{rmk}In \cite[Appendix B]{QZfin}, we   prove a stronger result about gonality of geometrically  connected components of Shimura curves
using a more analytic method.
 \end{rmk}
Now we can  promote  Theorem \ref{thm:2+}    to a statement akin to the one in  Conjecture \ref{conj:1}.
  
A simple abelian variety $A$ over a totally real field $F$ is called modular (following \cite{YZZ}) if 
  we have $ \leftidx{_B}\Pi_A\neq 0$ for some almost definite quaternion algebra $B$ over  $F$.
  (As an example,  an elliptic curve over $\BQ$  is modular 
by the modularity theorem of Wiles and Taylor--Wiles et. al. \cite{Wil,TW,BCDT}.)
We call $A$ \textit{of  discrete series  type} if it is modular for some almost definite quaternion algebra $B$ and there exists a finite place $v$ of $F$ such that the local component $ \leftidx{_B}\Pi_{A,v}$ is a discrete series. 
This definition does not depend on the choice of $B$.

 \begin{cor}\label{cor:2+} 
Let $A $  be a simple  modular   abelian variety over $F$  of  discrete series  type. 
 Let      $A^o$ be a simple isogeny factor of $A_{\BC}$.
   Then there exists a  curve $X/\BC$   such that  $h^{A^o}(X)\neq 0$ and  $[\Delta]_{h^{A^o}(X)}=0.     $ Moreover, among  non-hyperelliptic $X/\BC$  with $[\Delta]_{h^{A^o}(X)}=0$,  the integer 
 $g\lb h^{A^o}(X)\rb$  is unbounded.

   \end{cor}
   \begin{proof}
    Apply    Theorem \ref{thm:2+} to $A_1=A_2=A_3=A$ to get $E/\BQ$ and $D$.     Then   $\delta_{A_E,*}^{\otimes 3}[\Delta_L]=[\Delta_L]_{h^{A_{E}}(X_L)}= 0.$
Let $X$ be  a   connected component of $X_{L,\BC}$. 
By \eqref{eq:bc},
 $ \delta_{A^o} \circ \lb \delta_{A_E}|_{X^{2}}\rb=  \delta_{A^o} .$
 So  $[\Delta]_{h^{A^o}(X)}=0.
     $
Now we  prove the second part of the theorem. Let $g>0$.
    Choose a suitable $L\subset D^\times(\BA_{E,\rf})$ using Proposition \ref{prop:nonhyp}  (with $A,F,B,K ,d$ replaced by $A_E, E,D ,L,g\dim A $)   such that $X$ is not hyperelliptic and 
    $\dim_\BQ \Pi_{A_E}^L>g\cdot\dim A_E  \cdot |\pi_0 (X_{L,\BC}) |$.
Since $g\lb h^{A_E}(X_L)\rb=\dim_\BQ \Pi_{A_E}^L $ and 
   $$g\lb h^{A^o}(X)\rb\geq  g\lb h^{A_E}(X_L)\rb/(\dim A_E\cdot |\pi_0 (X_{L,\BC}) | ),$$  
 we have  $g\lb h^{A^o}(X)\rb>g$. 
        \end{proof}
 
  Theorem \ref{thm:2} is a special case of Corollary \ref{cor:2+}, as follows.

\begin{proof}[Proof of Theorem \ref{thm:2}] 
 
To avoid overusing the notation ``$E$", let us use $\BE$ for the elliptic curve in the theorem over $F=\BQ$.
Let $B=\GL_{2,\BQ}$. 
 As we mentioned above,  
by the modularity theorem  \cite{Wil,TW,BCDT},   
  we have $ \leftidx{_{B}} \Pi_\BE\neq 0$.  Moreover, the local $L$-factors of $\BE$ coincide with the ones of $\Pi_\BE $.
 As that the $j$-invariant of $\BE$ is non-integral at prime $p$, by the  reduction theory  \cite[Proposition 5.4, 5.5]{Sil} and  local $L$-factors  \cite[C.16]{Sil}  of elliptic curves,  $\Pi_{\BE,p} $ is a discrete series (in fact, a special representation). 
Now the theorem follows from   Corollary \ref{cor:2+}.  \end{proof}
  
  \begin{rmk}\label{rmk fun fd}The analog of Theorem \ref{thm:2} holds for any elliptic curve $E$ over the  function field $F$ of a curve $C$ over a finite field; 
   that is, for every $g>0$, there exists a   non-hyperelliptic curve $X/\ol F$   such that 
   $[\Delta]_{h^{E_{\ol F}}(X)}=0$  and $g\lb h^{E_{\ol F}}(X)\rb>g$.
 Indeed, there are two cases. First,  if $E$ is isotrivial, even the  analog of  Conjecture \ref{conj:1}  is already a special case of 
  a theorem of Soul\'e \cite[Theorem 3]{Sou}, which says on a product of curves  
 over a finite field,  Chow 1-cycles  with $\BQ_\ell$-coefficients coincide (via the cycle class map) with 
 Tate 1-cycles.
 Second, if
 $E$ is not isotrivial,   its $j$-invariant has a pole at some point of the proper $C$ (while for $\BQ$, the curve $\Spec \,\BZ$ is not ``proper").
Then the same proof as that of Theorem \ref{thm:2}  applies, using   \cite{Qiu} in place of \cite{YZZ}. 
Here, we can allow all function fields as the modularity of elliptic curves is known.
\end{rmk}

\section{``Modular" curves of genus 3}\label{curves of genus 3}

In this section, we provide examples of non-hyperelliptic  curves  with trivial automorphism groups and vanishing modified diagonal cycles.

\subsection{Curves by dominated by Shimura curves} % parameterized by Shimura curves?
Let $C$ be a  curve over  $F$.   
For  a quaternion algebra  $B$   over $F$,
 following \cite{BGGP}, we say that $C$ is $B$-modular if there is a Shimura curve $ X_K$   
for $K\subset B^\times(\BA_\rf)$ (as  in the last section) with  a non--constant morphism $ X_K\to C$. 
 Note that, for a given curve $C$ over $F$, there are only  finitely many quaternion algebras  $B$   over $F$ such that $C$ is $B$-modular.
In fact $ C$ must have bad reductions at places where $B$ is ramified. 

We have the following generalization of  \cite[Conjecture 1.1]{BGGP} which deals with case $B=\RM_{2,\BQ}$.
\begin{conj}
\label{c:mod} 
Given a quaternion algebra  $B$   over $F$  and $g>1$,  there are only  finitely many $B$-modular  curves of genus $g$.
% \item Given $C$, there are only  finitely many quaternion algebras  $B$   over $F$ such that $C$ is $B$-modular.
 \end{conj}
      We mention this conjecture only to better introduce this new notion ``modular"  curves.
      We refer to  \cite{BGGP} for further details.
      \subsection{Good modular curves}
      Let $C$ be $B$-modular.   
 We say that $C$ is $B$-good (simply good  if $B$  is clear from the context) if, for any  triple $(A_1,A_2,A_3) $ of 
simple   isogeny  factors of  $J_C$ (thus also of $J_{X_K}$),
 the $\BQ$-representation
 $  \Pi_{A_1}\otimes \Pi_{A_2}\otimes \Pi_{A_3}$ of $B^\times(\BA_\rf)$  has no nonzero  $B^\times(\BA_\rf)$-invariant linear forms.
  
The notion ``good curve" is useful   by the following lemma, easily deduced from \eqref{eq:equiv}, Proposition  \ref{vancor1}   and Corollary \ref {generalcor}.
  We say that a curve has vanishing  modified diagonal cycle if $[\Delta_e]=0$ for some $e$ as in \S \ref{1.2}, i.e., 
   if either its geometrically connected curve genus  $\leq 1$, or  
  $[
  \Delta_e]=0$  with  $e$ being   the canonical class divided by the degree as in \eqref {eq:xi}.
   \begin{lem}\label{vancor2} 
Good curves    have vanishing modified diagonal cycles.

\end{lem}

       \subsection{Good modular curves of genus $3$} 

Let $F=\BQ,B=\RM_{2,\BQ}$. 
 In \cite[Table 1]{GO}, some   $B$-modular  non-hyperelliptic genus 3 geometrically connected curves  $C$ are listed  in  \cite[Table 1]{GO}, with explicit projective equations.  
 Let us only consider $C/\BQ$ in  \cite[Table 1]{GO} dominated by the standard modular curve $ X_0(N)$. In particular $C$ is new. So  to each simple   isogeny  factor $A$ of  $\Pic^0_{C/F}$  is  associated a newform $f$ of $\Gamma_0(N)$.  And we say that  ``$f$ appears in $C$".

  The good ones are given in \eqref{sqfree},\eqref{p||N} and \eqref{459}. We only sketch a proof of their goodness, i.e., computations of trilinear forms. The reader may refer to \cite{QZfin}, where  we compute many examples, for  more details on such computations.

  If  
 $N$ is square free  in  \cite[Table 1]{GO}, that is  
 \begin{align*}
 N=&43,57,65,82,91,97,109,113,118,123,127, 139,\\ &141,149,151,179,187,203,
217,239,295,329,
\end{align*}  use Atkin--Lehner signs read from the database LMFDB \cite{LMFDB} to check if $C$ is good or not good. Explicitly, $C$ is good if and only if for 
every three (not necessarily different)  newforms $f_1,f_2,f_3$ appearing in $C$, there is some  prime $p|N$, such that the   Atkin--Lehner signs $s_{i,p}$ of $f_i$, $i=1,2,3$, satisfy   $\prod_{i=1}^3s_{i,p}=-1$. 
We find $C$ good if and only if 
\begin{equation}
\label{sqfree}C=C_{217}^A,C_{295}^A,C_{329}^C
\end{equation}
in the notations in \cite[Table 1]{GO}. 
 In general, if $p||N$, we similarly check whether $C$ is good (but we do not  to  check whether $C$ is not good). 
 We find $C$ good if   
 \begin{equation}
\label{p||N}C=C_{475}^E,C_{1175}^D.
\end{equation}

For $N= 99,169,369,  855=p^2M$ where the prime $p\nmid M$ (there is a unique such choice of $p$)  in  \cite[Table 1]{GO}, we   use 
 Prasad's dichotomy   \cite{Pra,Pra1}  and
 Jacquet--Langlands correspondence  at $p$ to the unit group $D^\times$ of the division quaternion algebra $D$ over $\BQ_p$. Indeed, the Jacquet--Langlands correspondences are  representations of dihedral groups (by the structure of $D^\times$). 
From the character tables of these groups, we easily prove that the Jacquet--Langlands correspondences does have trilinear forms (so the original one does not). 

The remaining $N$ are $N=243,459, 1215, 1539$.
We find $C$ good if and only if
\begin{equation}
\label{459}C=C_{459}^{B,I} .
\end{equation} 
The verification is by   using LMFDB,  the computer algebra system  Sage \cite{Sage} and a criterion of Prasad  \cite[Proposition 6.7]{Pra} (relating trilinear forms to compact induction datum). 
  And for  $N=1215$, the   computation by Sage is  pretty large.  We ran it with
 Apple M2 chips and 
 8 GB memory for more than
 17 hours!
   
   \begin{thm}\label{thm:AutC}
(1) The curves $ C_{217}^A,C_{295}^A,C_{329}^C ,C_{475}^E,C_{1175}^D $ and $C_{459}^{B,I}$, have vanishing modified diagonal cycles.

  (2) For $C=C_{217}^A,C_{295}^A,C_{329}^C ,C_{475}^E,C_{1175}^D $, $\Aut(C_\BC)$ is trivial.
 
 (3)   For $C= C_{459}^{B,I}$, $\Aut(C_\BC)<\BZ/2$.
   \end{thm}
   We need some preparations.
   \begin{lem}\label{lem:C1}   For $C=C_{217}^A,C_{295}^A,C_{329}^C ,C_{475}^E,C_{1175}^D$, its Jacobian $J$ is absolutely simple and $\End(J_\BC)_\BQ$ is a totally real cubic field.
\end{lem}
\begin{proof}By definition, $J$ is a modular abelian variety.  Thus $J$  is simple over $\BQ$ and $\End(J)_\BQ$ is a totally real  cubic field, generated by coefficients of (the Galois  orbit of) newform $f$  (see \cite[Table 1]{GO} for the labeling of the $f$)   corresponding to $\Pi_J$.  
By \cite[Proposition 2.1]{GL}, if the squares of $p$-th coefficients of   $f$, $p\nmid N$, generate $\End(J)_\BQ$, then $\End(J)_\BQ=\End(J_{\ol\BQ})_\BQ$ and the lemma follows.  
 But the ``if " condition follows trivially from that $\End(J)_\BQ$ is a totally real  cubic field.
\end{proof}
   
   \begin{lem}\label{lem:C2}     For $C_{459}^{B,I}$, its Jacobian $J=E\times A$ over $\BQ$ where $E$ is an elliptic curve  
   and $A$ is an abelian variety of dimension $2$ such that  $\End(E_\BC)_\BQ=\BQ$ and  $\End(A_\BC)_\BQ$  is a totally real quadratic field.
In particular, $\Hom(E,A)$ is trivial.
   \end{lem} 
\begin{proof} The proof is similar to  that of Lemma \ref{lem:C1}, except now we need to use LMFDB, which displays $\End(E_\BC)_\BQ$ and coefficients of (the Galois  orbit of) newform  corresponding to  $\Pi_A$.  
\end{proof}

\begin{proof}[Proof of Theorem \ref{thm:AutC}]
(1) follows form Lemma \ref {vancor2} and the discussion on the curves above the theorem.

(2) By  computations   as in our previous work  \cite{QZ1} (with   \cite[Table 1,2]{MSSV}), we find that a non-hyperelliptic genus 3   curve   over $\BC$ with simple Jacobian has automorphism group either contained in $\BZ/3$ or being $\BZ/9$.  However, 
by  \cite[Lemma 2.1]{BaHa}, $\Aut(C_\BC)$  is isomorphic to a (finite) subgroup of $\End(J _\BC)^\times$.
As $\End(J _\BC)$ is an order in $\End(J _\BC)_\BQ$, 
 $\Aut(X)$  is isomorphic to a   subgroup of $  \{\pm 1\}$ by Lemma \ref{lem:C1}.
  So  $\Aut(X)$  is trivial.

(3) is similar to (2), and we use Lemma \ref{lem:C2} instead. 
\end{proof}

\appendix
\section{Deninger--Murre projectors}
We collect a few known results on Fourier transform and Deninger--Murre projectors for (Chow groups of) abelian varieties.
\subsection{Fourier transform}
Let $F$ be a field and 	$S$  a smooth quasi-projective connected $F$-scheme. 
Let $A$ be an abelian scheme over $S$ and $\wh A$ the dual abelian scheme. 
 Let  $l\in \Ch^1(A\times_S \wh A ) $ be the class of a  Poincar\'e line bundle rigidified along the zero sections, see \cite[2.5]{DM}.
 For $a\in A(S)$, consider the closed embedding $$i_a=a\times \id_{\wh A}: \wh A=S\times_S \wh A\to A\times_S \wh A .$$
%  \subsection{Fourier}
Let $\Ch^*(A\times_S \wh A )$ be the Chow ring  with the usual intersection product. 
Define 
  $$f=\exp(l):=\sum_{i\geq 0} \frac{l^n}{n!}\in \Ch^*(A\times_S \wh A ) .$$
  %The biduality of  $A$ implies that the Fourier transform from  $\wh A$ to $  A $ is  the correspondence$$\cF^t\in  \Ch^*(\wh A\times_S  A ) .$$
   This is the kernel to define the Fourier transform from  $ \Ch^*(A)$ to $ \Ch^*( \wh A )$. We have  the following equation of Beauville 
   in the relative case \cite[(1.3.3)]{Kun}
      \begin{equation}\label{faS} 
  f_*(a_* [S])=\exp(i_a^*l) .
   \end{equation} 
   \subsection{Deninger--Murre projectors }
For $n\in \BZ$, let   ${n_A}$ be the  ``multiplication by $n$" morphism on $A$.
For $j\in \BZ$, consider the following eigenspace $$\Ch^i_j(A)=\{z\in\Ch^i(A):\forall n\in \BZ, n_A^* z=n^{2i-j}z \}.$$
By \cite[Theorem  2.19]{DM}, $\Ch^i_j(A)\neq 0$ for  only finitely many $j$, and 
we have 
$$\Ch^i(A)=\mathop\oplus_{j\in \BZ}\Ch^i_j(A).$$
 The Deninger--Murre projectors  are projectors for this decomposition, simultaneously for all $i$.
  \begin{thm}[{\cite[Theorem 3.1]{DM}}]\label{Deninger--Murre projectors}
  Let $\Gamma_{n_A}\subset A\times_S  A $ be the graph of $n_A.$
There is a unique decomposition $$[\Gamma_{1_A}]=\sum_{k=0}^{2g}\pi_k\quad \text{in}\quad\Ch^*(A\times_S   A )$$
such that $
 [\Gamma_{{n_A}}^t]\circ \pi_k=n ^k\pi_k =
 \pi_k\circ [\Gamma_{{n_A}}^t] $ for all $n\in \BZ$. Moreover, $\pi_k\circ\pi_j=0$ if $i\neq j$, and $\pi_k^2=\pi_k$.

\end{thm}
  \begin{cor} \label{Deninger--Murre projectorscor}
(1) We have $\pi_k^t=\pi_{2g-k}.$ In particular, $\pi_k\circ [\Gamma_{{n_A}}]= n ^{2g-k}\pi_k$.

(2) Let $u$ be the identity $S$-point of $A$.
If $k\neq 2g$, $\pi_{k,*}[u]=0$

(3) If $S=\Spec F$, then $\pi_{1,*}|_{\Ch^1(A)_0}$ is identity.

\end{cor}
   \begin{proof} 
(1)  is  the 3rd remark after \cite[Theorem 3.1]{DM}.

(2) follows from (1) and that $[\Gamma_{{n_A},*}]u=n^{2g}u$.

(3)  follows from Theorem \ref{Deninger--Murre projectors} and  the well-known result ${\Ch^1(A)_0}=\Ch^1_1(A)$. \end{proof}

   \subsection{Relation}
   \begin{lem}\label{Deninger--Murre projectors1}
 For $a\in A(S)$ and $j=0,...,2g$, we have  $$f_*(\pi_{2g-j,*}(a_* [S]))=\frac{( i_a^*l)^j}{j} \quad \text{in}\quad\Ch^j(\wh  A ).$$ 
\end{lem}
     \begin{proof} By Theorem \ref{Deninger--Murre projectors},  $\pi_{2g-j,*}(a_* [S])\in \Ch^g_j(A)$.    By \cite[Lemma 2.18]{DM}, $f_*(\pi_{2g-j,*}(a_* [S]))\in \Ch^j(\wh A)$. Taking sum over $j$'s and comparing with \eqref{faS}, the lemma follows.
     \end{proof}
    The biduality theorem for abelian schemes   implies that  
 \begin{equation}\label{a+b} 
 i_a^*l+ i_b^*l= i_{a+b}^*l,\quad \forall a, b \in A(S).
 \end{equation} 
Then by Lemma \ref{Deninger--Murre projectors1}, \eqref{a+b} and   Fourier inversion \cite[Corollary 2.22]{DM}, we have the following.
   \begin{cor}\label{Deninger--Murre projectors2}
 For $a,b\in A(S)$, we have  $$ \pi_{2g-1,*}(a_* [S])+\pi_{2g-1,*}(b_* [S])=\pi_{2g-1,*}((a+b)_* [S]).$$ 
\end{cor}

\end{document}